\documentclass[10pt]{article}

\usepackage{amsmath}
\usepackage{amsthm}
\usepackage{amsfonts}
\usepackage{amssymb}
\usepackage{mathtools}
\usepackage{mathrsfs}
\usepackage{stmaryrd}
\usepackage{enumitem}
\usepackage{tikz-cd}
\usepackage[english]{babel}
\usepackage{cite}
\usepackage{hyperref}

\newcommand{\for}{\hspace{0.5pt} :\hspace{0.5pt}} 
\renewcommand{\subset}{\subseteq}
\renewcommand{\supset}{\supseteq}
\DeclareMathOperator{\ima}{im}

\newcommand{\clos}[1]{\overline{#1}}
\DeclareMathOperator{\intraaux}{int}
\newcommand{\intra}[1]{\intraaux\left(#1\right)}

\newcommand{\real}{\mathbb{R}}
\newcommand{\field}{\mathbb{K}}
\newcommand{\rational}{\mathbb{Q}}
\newcommand{\nat}{\mathbb{N}}

\renewcommand{\leq}{\leqslant}
\renewcommand{\geq}{\geqslant}
\newcommand{\abs}[1]{\left\lvert #1 \right\rvert}

\newcommand{\grad}{\nabla}
\newcommand{\scalar}{\hspace{1pt}\vert\hspace{1pt}}

\newcommand{\tangent}[2]{T_{#1}(#2)}

\DeclareMathOperator{\supp}{supp}
\newcommand{\formser}[2]{#1 \llbracket #2 \rrbracket}
\newcommand{\genser}[2]{#1 \llbracket #2^* \rrbracket}
\newcommand{\mixser}[3]{#1\llbracket #2^*, #3 \rrbracket}

\newcommand{\graph}[1]{\Gamma(#1)}
\newcommand{\germalg}[2]{\mathcal{A}_{#1, #2}}
\newcommand{\gengermalg}[1]{\mathcal{A}_{#1}}
\newcommand{\functionalg}[3]{\mathcal{A}_{#1, #2, #3}}
\newcommand{\genfunctionalg}[2]{\mathcal{A}_{#1, #2}}
\newcommand{\genpolydisk}[2]{I_{#1, #2}}
\newcommand{\polydisk}[3]{I_{#1, #2, #3}}
\DeclareMathOperator{\taylor}{\mathbf{T}}

\renewcommand{\phi}{\varphi}
\renewcommand{\epsilon}{\varepsilon}
\DeclareMathOperator{\trdeg}{trdeg}

\newcommand{\almostsmooth}{\mathcal{V}}
\newcommand{\norm}[2]{\left \rVert #1 \right \rVert_{#2}}

\theoremstyle{definition}
\newtheorem{defn}{Definition}[section]
\newtheorem{remark}[defn]{Remark}
\newtheorem{digr}[defn]{}
\theoremstyle{plain}
\newtheorem{thm}[defn]{Theorem}
\newtheorem{prp}[defn]{Proposition}
\newtheorem{cor}[defn]{Corollary}
\newtheorem{lem}[defn]{Lemma}
\newtheorem{assumption}[defn]{Assumption}

\title{Generalized quasianalytic algebras with weakly smooth germs generate
  o-minimal structures}
\author{Rémi Guénet}

\begin{document}
\maketitle
\tableofcontents

\section*{Introduction}

In \cite{rolin-servi-monomialization}, Rolin and Servi give a proof of
o-minimality and model-completeness under general assumptions
satisfied by most of the currently known polynomially bounded o-minimal
expansions of the real field. However, all structures that fit the setting of
\cite{rolin-servi-monomialization} must have smooth
cell-decomposition. The goal of this paper is to weaken the assumptions of
this result so that it may apply to structures that do not have smooth
cell-decomposition. As a result, we obtain an axiomatic way to
construct new o-minimal structures which is sufficiently general to
apply to all of the currently known polynomially bounded o-minimal
expansions of the real field. We also use our result to construct a
new o-minimal structure within which we can define a nowhere smooth
function. \\

Throughout this paper, we will be using various model-theoretic
notions which we explain here in a somewhat more geometric way so as
to make our results more widely accessible. Let \(\mathcal{F}\) be a
collection of functions \(f \colon \real^m \to \real\) (here \(m \in
\nat\) is allowed to depend on \(f\)). We will say that
\(\real_\mathcal{F} = \langle \real; <, +, -, \times, 0, 1,
\mathcal{F} \rangle\) is the \textbf{expansion of the real field by the
  functions in \(\mathcal{F}\).} The collection of
\(\real_\mathcal{F}\)\textbf{-definable} sets is the smallest collection
\(\mathcal{D} = (\mathcal{D}_n)_{n \geq 1}\) such that
\begin{itemize}
\item \(\mathcal{D}_n\) is a collection of subsets of \(\real^n\)
  containing all semi-algebraic subsets of \(\real^n\) and the graphs
  of all functions \(f \colon \real^{n-1} \to \real\) in
  \(\mathcal{F}\). 
\item \(\mathcal{D}\) is closed under binary unions, binary
  intersections, taking set complements and projections. 
\end{itemize}
We say that a set \(A \subset \real^n\) is
\(\real_\mathcal{F}\)-definable if \(A \in \mathcal{D}_n\). We say
that a function \(f \colon \real^n \to \real\) is
\(\real_\mathcal{F}\)-definable if its graph is
\(\real_\mathcal{F}\)-definable. If the structure
\(\real_\mathcal{F}\) is understood, we will simply say ``definable''
instead of ``\(\real_\mathcal{F}\)-definable''. We will say that the
structure \(\real_\mathcal{F}\)
\begin{itemize}
\item is \textbf{model-complete} if one does not need to take complements
  in order to generate the whole family of definable sets.
\item is \textbf{o-minimal} if every definable set has only finitely many
  connected components. 
\item has \textbf{quantifier-elimination} if one does not need to take
  projections in order to generate the whole family of definable
  sets. 
\end{itemize}
Notice that, whenever \(\real_\mathcal{F}\) has
quantifier-elimination, it is necessarily model-complete. \\

Let us now explain briefly the setting of
\cite{rolin-servi-monomialization}. The authors consider a collection
of functions \(\mathcal{A}\) and they assume that, to each function
\(f \in \mathcal{A}\), we can associate a series with non-negative
real exponents, written \(\taylor(f)\), which should be thought of as
an asymptotic expansion of \(f\) at \(0\). The key hypothesis about
these expansions is known as \textbf{generalized quasianalyticity} and
it states that any two functions in \(\mathcal{A}\) with the same
asymptotic expansion must coincide in a neighborhood of \(0\). Since
the functions in \(\mathcal{A}\) are not all smooth, we need to resort
to using series with non-negative real exponents. However, in the
smooth case, the asymptotic developments will simply be given by the
Taylor series. In this case, we simply say ``quasianalyticity''
instead of ``generalized quasianalyticity''.

Under this hypothesis, along with various properties of the class
\(\mathcal{A}\) of functions (such as being stable under composition,
partial derivatives and the Implicit Functions Theorem), the authors
then prove that the structure \(\langle \real; <, 0, 1, +, -, \times,
\mathcal{A} \rangle\), which is called \(\real_\mathcal{A}\), is both
o-minimal and model-complete. We are now going to outline the method
used. 

Series with non-negative real exponents are called 
\textbf{generalized power series}. In \cite{rolin-servi-monomialization},
the authors prove a monomialization (or desingularization) result for
generalized power series in order to obtain corresponding results for the 
functions in \(\mathcal{A}\). This desingularization result allows the
authors to obtain parametrization theorems for sets defined by
equations and inequations involving functions in \(\mathcal{A}\). When
combining this with a Fiber Cutting Lemma (see
\cite{rolin-servi-monomialization}, 3.11) and Gabrielov's Theorem of the
Complement (see Theorem 2.7 in
\cite{vdDries-Speissegger-genealized-power-series}), the authors
are able to prove that \(\real_\mathcal{A}\) is both o-minimal and
model-complete. This method had been used
before to prove o-minimality in many special cases such as in
\cite{vddries-generalization-tarski-seidenberg,vddries-speissegger-gevrey-functions,vdDries-Speissegger-genealized-power-series,rolin-speissegger-wilkie-denjoy-carleman-classes,kaiser-rolin-speissegger-non-resonant}. \\

However, there is one notable example of an o-minimal and polynomially bounded
structrure which does not satisfy the assumptions of
\cite{rolin-servi-monomialization}, namely the structure \(\real_H\)
introduced in \cite{legal-rolin-not-c-infty} by Le Gal and
Rolin. Indeed, this
structure was the first example of an o-minimal structure that does
not have smooth cell-decomposition, while it turns out that the
assumptions of \cite{rolin-servi-monomialization} imply that the
structures considered must always have smooth cell-decomposition.
Despite that, we have already proven in \cite{guenet-weakly-smooth}
that the theorems of \cite{rolin-servi-monomialization} apply also to
the structure \(\real_H\). The goal of this paper is to prove
o-minimality and model-completeness in an axiomatic setting that
generalizes both \cite{rolin-servi-monomialization} and
\cite{guenet-weakly-smooth}. 

In order to do so, we reuse the setting of
\cite{rolin-servi-monomialization}, except that we replace
\(\mathcal{A}\)-analyticity (see \cite[Definition
1.10]{rolin-servi-monomialization}) with a weaker assumption given in
\ref{graph-simple}. As we saw already in \cite{guenet-weakly-smooth},
most of the local results in \cite{rolin-servi-monomialization} can
still be obtained without any significant change. However, it is much
more difficult to derive global results.
The main difference between this paper and \cite{guenet-weakly-smooth}
is that, in this paper, asymptotic developments may be given by
generalized power series whereas formal power series were sufficient in the
case of \(\real_H\). 

In Section
\ref{sec:monomialization}, we recall the
main steps of the monomialization algorithm from
\cite{rolin-servi-monomialization}. For later purposes, we need a
slight strengthening of their Theorem 2.11. We obtained such a
strengthening in \cite[Theorem 3.11]{guenet-weakly-smooth} in the case
of formal power series. Furthermore, we have no trouble adapting this
theorem to the case of generalized power series by appealing to various
results from \cite{rolin-servi-monomialization}.

Section \ref{sec:geometric} is dedicated to the geometric arguments
needed to prove Theorem \ref{thm:o-minimal}, according to which
\(\real_\mathcal{A}\) is o-minimal
and model-complete. In Paragraph \ref{par:basic-defs}, we introduce
the classical definitions needed in such geometric arguments. Then, in
Paragraph \ref{par:simple-sub-sets}, we define the new notion of
\textbf{simple sub-\(\Lambda\)-sets}. This notion is crucial in
stating Assumption \ref{graph-simple}, which replaces
\(\mathcal{A}\)-analyticity (see \cite[Definition
1.10]{rolin-servi-monomialization}). 

Finally, in Paragraph \ref{par:parametrizations}, we prove
parametrization theorems while Paragraph \ref{par:fiber-cutting}
focuses on obtaining a Fiber Cutting Lemma. These results combine to
give Theorem
\ref{sub-sets-are-simple}, according to which the assumptions of
Gabrielov's Theorem of the Complement (see \cite[Theorem
2.7]{vdDries-Speissegger-genealized-power-series}) are satisfied. It
turns out that Gabrielov's Theorem of the Complement implies that the
\textbf{sub-\(\Lambda\)-sets}, introduced in Paragraph
\ref{par:basic-defs}, are exactly the bounded definable sets. In view
of their definition, it is easy to deduce
model-completeness. Furthermore, o-minimality follows at once
from Corollary \ref{global-parametrization} according to which
sub-\(\Lambda\)-sets have finitely many connected components. 

In Section 4 of \cite{rolin-servi-monomialization}, the authors also
obtain a quantifier-elimination result for the expansion \(\langle
\real_\mathcal{A};
\frac{1}{x}, (\sqrt[n]{x})_{n \in \nat}\rangle\) where the new symbols are
interpreted in the obvious way wherever they are defined and extended
by \(0\) outside their domain of definition. As one can check, this
result makes use of the o-minimality of \(\real_\mathcal{A}\), but it
does not depend on \(\mathcal{A}\)-analyticity so that it extends to
the setting of this paper. 

Finally, in Section \ref{sec:example}, we construct a new o-minimal
structure using Theorem \ref{thm:o-minimal}. The striking feature of
this structure is that we can define a nowhere smooth function in it.
The idea of this construction was suggested by Olivier Le Gal and
relies heavily on ideas he introduced in \cite{legal-generic}. 

\section{Monomialization}
\label{sec:monomialization}

\subsection{Generalized power series}

This paragraph is a short reminder on the properties of generalized
power series. More details can be found in \cite[section
4]{vdDries-Speissegger-genealized-power-series}. Firstly, we say
that \(S \subset [0, \infty)^m\) is a \textbf{good set} when there are well
ordered subsets \(S_1, \dots, S_m \subset [0, \infty)\) such that \(S
\subset S_1 \times \dots \times S_m\). If \(\alpha, \beta \in [0,
\infty)^m\), we write \(\alpha \leq \beta\) to mean that \(\alpha_i
\leq \beta_i\) for every \(1 \leq i \leq m\). Given a good set \(S \subset [0,
\infty)^m\), we define \(S_\text{min}\) to be the set of
minimal elements of \(S\). It is proven in
\cite[Lemma 4.2.1]{vdDries-Speissegger-genealized-power-series} that 
\(S_{\text{min}}\) is finite and that, for every \(\alpha \in S\),
there exists \(\beta  \in S_\text{min}\) such that \(\beta \leq
\alpha\).

Consider \(X = (X_1, \dots, X_m)\) a tuple of variables. Then, a
\textbf{generalized power series} is a formal series \(F = \sum_\alpha c_\alpha
X^\alpha\) where \(\alpha\) ranges over \([0, \infty)^m\) and such
that the support \(\supp F = \{\alpha \for c_\alpha \neq 0\}\) is a
good set. The requirement on the support guarantees that these series
can be added and multiplied in the usual way. Thus, the set of all
such series is a ring that we write \(\genser \real X\).

If \(\mathcal{F} \subset \genser \real X\), then we let
\[\supp \mathcal{F} = \bigcup_{F \in \mathcal{F}} \supp F\]
and we say that \(\mathcal{F}\) has \textbf{good total support} when \(\supp
\mathcal{F}\) is a good set.

Finally, if \(Y = (Y_1, \dots, Y_n)\) is a second tuple of variables,
we write \(\mixser \real X Y\) for the set of all series \(F \in
\genser \real {(X, Y)}\) such that \(\supp F \subset [0, \infty)^m
\times \nat^n\). This is a subring of \(\genser \real {(X, Y)}\), we say
that its elements are \textbf{mixed series}. 

\subsection{Quasianalytic algebras}

This paragraph is largely based on
\cite[section 1.2]{rolin-servi-monomialization}. There are a few
differences in the notation which will be pointed out when they are
introduced.

Firstly, let \(m, n \geq 0\) be integers. Then, a \textbf{polyradius} is a tuple
of the form \(r = (s_1, \dots,
s_m, t_1, \dots, t_n) \in (0, \infty)^{m+n}\). The set
\[\polydisk m n r = [0, s_1) \times \dots \times [0, s_m) \times
  (-t_1, t_1) \times \dots \times (-t_n, t_n)\]
is called a \textbf{polydisk}. In \cite{rolin-servi-monomialization},
the same polydisk is written \(\hat{I}_{m, n,
  r}\), while \(\polydisk m n r\) refers to its interior. When \(n
= 0\), we will write
\(\genpolydisk m r\) instead of \(\polydisk m n r\). 

Given \(m, n \geq 0\) two integers and \(r\) a polyradius, we fix an
algebra \(\functionalg m n r\) of continuous functions \(f \colon
\polydisk m n r \to \real\). Throughout the document, we will assume
that \(\functionalg m n r\) satisfies points (1)-(8) of
\cite[1.8]{rolin-servi-monomialization}. However, we will not be
assuming
that the functions \(f \in \functionalg m n r\) are \(C^1\) on
\(\intra{\polydisk m n r}\). When \(n = 0\), we might write
\(\genfunctionalg m r\) in place of \(\functionalg m n r\). 

As in \cite{rolin-servi-monomialization}, we write
\(\germalg m n\) for the algebra of germs at \(0\) of the functions in
\(\functionalg m n r\). When \(n=0\), we write more
simply \(\gengermalg m\) instead of \(\germalg m n\). As in
\cite[Definition 1.11]{rolin-servi-monomialization}, we assume that
there are injective algebra morphisms \(\taylor_{m, n} \colon \germalg
m n \to \mixser \real X Y\) and that \(\taylor_{m', n'}\) extends
\(\taylor_{m, n}\) whenever \(m'+n' \geq m+ n\) and \(m' \geq m\). 
Since this last assumption removes any possible
ambiguity, we will write simply \(\taylor\) instead of \(\taylor_{m,
  n}\). Also, if \(f \in \germalg m n\), we let \(\widehat f =
\taylor(f)\). If \(\alpha \in [0, \infty)\), we
say that \(\alpha\) is an admissible exponent when there are \(m, n
\in \nat\), \(f \in \germalg m n\) and \(\beta \in \supp \widehat f\)
such that \(\alpha\) is a component of \(\beta\). We let
\(\mathbb{A}\) be the semi-ring generated by all admissible exponents
and we let \(\field\) be the field generated by \(\mathbb{A}\). 

We are not going to recall the definition of blow-up charts here, it
can be found in \cite[Definition
1.13]{rolin-servi-monomialization}. Henceforth, we assume that
conditions (1)-(7) in \cite[1.15]{rolin-servi-monomialization}
are satisfied. By \cite[Remark 1.17]{rolin-servi-monomialization},
this implies in particular that, if \(f \in 
\functionalg m n r\), then the germs at \(0\) of \(\frac{\partial
  f}{\partial y_j}\) and \(x_i \frac{\partial f}{\partial x_i}\) are
in \(\germalg m n\) for \(1 \leq i \leq m\) and \(1 \leq j \leq n\). 
As a consequence, if \(f \in \functionalg m n r\) and \(k \geq 0\) is
an integer, there exists a polyradius \(r' \leq r\) such that \(f
\restriction \intra {\polydisk m n {r'}}\) is \(C^k\). However, contrary
to \cite{rolin-servi-monomialization}, it is not necessarily the case
that \(f \restriction \intra {\polydisk m n r}\) is smooth because we
are not assuming \(\mathcal{A}\)-analyticity (see \cite[Definition
1.10]{rolin-servi-monomialization}). 

\subsection{Monomialization}

This paragraph and the next are a brief reminder on the monomialization
results from \cite[section 2]{rolin-servi-monomialization}. We will
introduce slight modifications to the results but they do not change
the proofs in any significant way. First of all, we modify the
definition of elementary transformations to include reflections.

\begin{defn}
  An \textbf{elementary transformation} is a map \(\nu \colon
  \polydisk{m'}{n'}{r'} \to \polydisk m n r\) that is of one of the
  following types.
  \begin{itemize}
  \item It is of one of the types described in \cite[Definition
    2.1]{rolin-servi-monomialization};
  \item We have \(m' = m+1, n' = n-1\) and \(\nu = \sigma_{m+i}^\pm\)
    for some \(1 \leq i \leq n\) where
    \[\sigma_{m+i}^\pm(x', y') =
      \begin{cases}
        x_k = x_k' &1 \leq k \leq m \\
        y_k = y_k' &1 \leq k < i\\ 
        y_i = \pm x_{m+1}' \\
        y_k = y_{k-1}' &i < k \leq n
      \end{cases}\]
  \end{itemize}
  The maps of the second type are called \textbf{reflections}. A finite
  composition of elementary transformations is called an \textbf{admissible
  transformation}. 
\end{defn}

Recall from \cite[Lemma 2.5]{rolin-servi-monomialization} that each
elementary transformation \(\nu\) induces an injective algebra
homomorphism \(\mixser \real X Y \to \mixser \real {X'} {Y'}, F \mapsto F
\circ \nu\). Since we have changed the notion of elementary
transformation, we also need to change the notion of elementary tree.

\begin{defn}
  An \textbf{elementary tree} is either a tree that has one of the forms described in
  \cite[Definition 2.6]{rolin-servi-monomialization} or it is a tree
  of the form
  
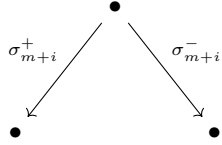
\begin{center}
    \begin{tikzcd}
      &\bullet \arrow[']{ddl}{\sigma^+_{m+i}}
      \arrow{ddr}{\sigma^-_{m+i}}\\
      \\
      \bullet & &\bullet
    \end{tikzcd}    
  \end{center}
  for some \(1 \leq i \leq n\). \textbf{Admissible trees} of height at
  most \(h\) are defined inductively on the ordinal \(h\). The only
  admissible tree of height \(0\) is the tree which is reduced to its
  root. An admissible tree of height at most \(h\) is an elementary
  tree with admissible trees of height \(<h\) attached to each of its
  leaves. 
\end{defn}

\begin{remark}
  Notice that we have allowed the height of an admissible tree \(T\)
  to be an ordinal. However, each branch of \(T\) must be finite
  whence it induces an admissible transformation \(\rho\).
\end{remark}

\begin{defn}
  Consider \(f(x, y) \in \germalg m n\), we say that \(f\) is \textbf{normal}
  if there exists \(\alpha \in \mathbb{A}^m\)
  and \(\beta \in \nat^n\) as well as
  \(u \in \germalg m n\) such that \(u(0)\neq 0\) and \(f(x, y) =
  x^\alpha y^\beta u(x, y)\). We also say that a series \(F \in
  \mixser \real X Y \cap \ima \taylor\) such that \(F \neq 0\) is
  \textbf{normal} when there are 
  \(\alpha \in \mathbb{A}^m\) and \(\beta \in \nat^n\) as well as a
  unit \(U \in \mixser \real X Y\) such that \(F = X^\alpha Y^\beta
  U\). 
\end{defn}

\begin{remark}
  If \(f \in \germalg m n\) then \(f\) is normal if and only if
  \(\widehat f\) is normal. 
\end{remark}

According to \cite[Theorem 2.11]{rolin-servi-monomialization}, we
have the following result.

\begin{thm}
  If \(F_1, \dots, F_q \in \mixser \real X Y \cap \ima \taylor\), then
  there is an admissible
  tree \(T\) such that, if \(\rho\) is an admissible transformation
  induced by one of the branches of \(T\), then \(F_1 \circ \rho,
  \dots, F_q \circ \rho\) are all normal. \qed
\end{thm}

If \(T\) is a tree as in the statement of the theorem above then we
say that \(T\) \textbf{monomializes} the series \(F_1, \dots, F_q\). For
reasons that
will become clear later, this result is not quite sufficient for our
purposes. In order to state the desired result, we are first going to
need a few definitions.

\begin{defn}
  \label{defn:critical-variable}
  Let \(\nu\) be a blow-up chart. Then, there is exactly one variable
  \(W\) by which we need to divide in order to write the inverse of
  \(\nu\). We call it the \textbf{critical variable} of the
  blow-up. Given \(F_1, \dots, F_q \in \mixser \real X Y \cap \ima
  \taylor\) and an admissible tree \(T\), we say that \(T\)
  \textbf{\(*\)-monomializes} the series \(F_1, \dots, F_q\) if
  \begin{itemize}
  \item \(T\) monomializes the series \(F_1, \dots, F_q\). 
  \item If \(\nu\) is a blow-up with critical variable \(W\) that
    occurs in \(T\) and if \(T'\) is the sub-tree of \(T\) below
    \(\nu\), then \(T'\) monomializes \(W\). 
  \end{itemize}
\end{defn}

\begin{remark}
  Using \cite[Lemma 2.9]{rolin-servi-monomialization} is easy to see
  that an admissible tree \(T\) \(*\)-monomializes the series
  \(F_1, \dots, F_q\) if and only if \(T\) \(*\)-monomializes the
  product \(F_1 \dots F_q\). 
\end{remark}

The required theorem is then the following.

\begin{thm}
  \label{*-monomialization}
  Given \(F_1, \dots, F_q \in \mixser \real X Y \cap \ima \taylor\),
  there exists an admissible tree \(T\) that \(*\)-monomializes the
  functions \(F_1, \dots, F_q\). 
\end{thm}

\begin{proof}
  This theorem is proven in detail as \cite[Theorem
  3.11]{guenet-weakly-smooth} in case all the variables are
  standard. Thus, we only need to explain how to modify the proof in
  order to take into account generalized variables. It turns out that
  we need to modify the method to make the series regular in \(Y_n\)
  but this is the only real difference. 
  
  By \cite[Lemma 2.9]{rolin-servi-monomialization}, we can assume that
  \(q=1\) and we write \(F = F_1\). Also, by \cite[Lemma
  2.13]{rolin-servi-monomialization}, we may assume that \(F(0, Y)
  \neq 0\). Doing as in the beginning of the proof of \cite[Theorem
  2.11]{rolin-servi-monomialization}, we may even assume that \(F\) is
  regular in the variable \(Y_n\). From here, we can do everything in
  the same way as in the proof of \cite[Theorem
  3.11]{guenet-weakly-smooth}.
\end{proof}

\section{O-minimality and model completeness}
\label{sec:geometric}

\subsection{Some definitions}
\label{par:basic-defs}

For each function \(f \in \functionalg m n r\), we define a total
function \(\widetilde f \colon \real^{m+n} \to \real\) such that
\(\widetilde f(x, y) = f(x, y)\) if \((x, y) \in \polydisk m n r\) and
\(\widetilde f(x, y) = 0\) otherwise. Then, we consider the language
\(\mathcal{L}_{\mathcal{A}}\) which is the language of ordered rings
\(\{<, 0, 1, +, -, \cdot\}\) augmented with a function symbol for each
function \(\widetilde f\). We let \(\real_\mathcal{A}\) be the real
ordered field with its natural \(\mathcal{L}_\mathcal{A}\)
structure. We will show in this section that the structure
\(\real_\mathcal{A}\) is o-minimal and model complete under an
assumption given in \ref{graph-simple}. To this end, we
will be using Gabrielov's approach (see \cite[Corollary
2.9]{vdDries-Speissegger-genealized-power-series}).

Here are the necessary definitions. Firstly, if \(\genpolydisk m r\subset
\real^m\) is a polydisk and \(f, g_1, \dots, g_q \in \genfunctionalg m
r\) then the set
\[B = \{x \in \genpolydisk m r \for f(x) = 0, g_1(x) > 0, \dots, g_q(x) >
  0\}\]
is called \textbf{\(\mathcal{A}\)-basic}. If \(A \subset \genpolydisk m r\) is a
finite union of \(\mathcal{A}\)-basic sets then we say that
\(A\) is a \textbf{\(\mathcal{A}\)-set}. Notice in particular that the class of
\(\mathcal{A}\)-sets is closed under finite unions and intersections.

Now, consider \(A \subset \real^m\). We say that \(A\) is
\textbf{\(\mathcal{A}\)-semianalytic} when, for all \(a \in \real^m\) and all
\(\sigma \in \{-1, 1\}^m\), there is a polydisk \(\genpolydisk m r\) such
that the set \(h_{a, \sigma}(A) \cap \genpolydisk m r\) is a
\(\mathcal{A}\)-set where
\[h_{a, \sigma}(x) = (\sigma_1(x_1 - a_1), \dots, \sigma_m(x_m -
  a_m)).\]
We say that \(A\) is a \textbf{\(\Lambda\)-set} when it is
\(\mathcal{A}\)-semianalytic and bounded. Finally, we say that \(A\)
is a \textbf{sub-\(\Lambda\)-set} when there is a \(\Lambda\)-set \(A' \subset
\real^{m+k}\) for some integer \(k \geq 0\) such that \(A =
\Pi_m(A')\). \\

When we say \textbf{manifold}, we will mean a \(C^1\)-submanifold of
\(\real^m\) for some integer \(m \geq 0\). As in
\cite[4379]{vdDries-Speissegger-genealized-power-series}, we will say
that a set \(S \subset \real^m\) has dimension when it is a countable
union of manifolds. In this case, we set
\[\dim(S) = \max\{\dim(M) \for M \subset S \text{ is a manifold}\}\]
when \(S \neq \varnothing\) and \(\dim(\varnothing) =
-\infty\). Recall from
\cite[4379]{vdDries-Speissegger-genealized-power-series} the following
facts.
\begin{itemize}
\item If \(S\) is a manifold then the dimension defined above agrees
  with its dimension as a manifold. 
\item If \(S = \bigcup_{j \in J} S_j\) where \(J\) is a countable set
  and each \(S_j\) has dimension then \(S\) also has dimension and
  \(\dim(S) = \max\{\dim(S_j) \for j \in J\}\). 
\item If \(M \subset \real^m\) is a manifold and \(f \colon M \to
  \real^n\) is a \(C^1\)-map of constant rank \(r\) then \(f(M)\) has
  dimension and \(\dim(f(M)) = r\). 
\end{itemize}

The goal of this section is to prove that the structure
\(\real_\mathcal{A}\) is o-minimal. In order to do so, we want to show
that every definable set has finitely many connected component. Since
we can define an homeomorphism from \(\real\) to the interval \((-1,
1)\) in \(\real_\mathcal{A}\), it suffices to prove that all bounded
definable sets have finitely many connected components. We will use
Gabrielov's Theorem of the Complement (see \cite[Theorem
2.7]{vdDries-Speissegger-genealized-power-series}) to prove that the
collection of sub-\(\Lambda\)-sets is stable under set
difference. From this result, it is easy to deduce that
sub-\(\Lambda\)-sets are exactly the bounded definable
sets (see \cite[Theorem 4.18]{guenet-weakly-smooth}). Furthermore, in
Corollary
\ref{global-parametrization}, we prove that every
\(\Lambda\)-set has finitely many connected components. The same must
be true about sub-\(\Lambda\)-sets since they are themselves
continuous images of \(\Lambda\)-sets. Finally, in order to prove that
the assumptions of Gabrielov's Theorem of the Complement are
satisfied, we prove in Theorem \ref{sub-sets-are-simple} that, for
every sub-\(\Lambda\)-set \(A \subset \real^m\), there exist an
integer \(k \geq 0\) and a \(\Lambda\)-set \(A' \subset \real^{m+k}\)
such that \(\Pi_m(A') = A\) and \(\Pi_m \restriction A'\) has finite
fibers. 

\subsection{Simple sub-\(\Lambda\)-sets}
\label{par:simple-sub-sets}

Considering the goal of this section, we will be especially interested
in the following class of sub-\(\Lambda\)-sets.

\begin{defn}
  A set \(A \subset \real^m\) is called a \textbf{simple sub-\(\Lambda\)-set}
  when there is a \(\Lambda\)-set \(A' \subset \real^{m+k}\) such that
  \(A = \Pi_m(A')\) and \(\Pi_m \restriction A'\) has finite fibers. 
\end{defn}

The goal of this section is then to prove that every
sub-\(\Lambda\)-set is simple (see Theorem \ref{sub-sets-are-simple}).
The lemma below summarizes the main properties enjoyed by the class of
simple sub-\(\Lambda\)-sets. All of these properties are immediate and
we will thus omit the proof. 

\begin{lem}
  \label{simple-lem}
  Consider \(A, B \subset \real^m\) and \(C \subset \real^n\) three
  simple sub-\(\Lambda\)-sets. Then, the sets \(A \cap B, A \cup B\)
  and \(A \times C\) are also simple sub-\(\Lambda\)-sets. If \(k \leq
  m\) and \(x \in \real^k\) then the fiber
  \[\{y \in \real^{m-k} \for (x, y) \in A\}\]
  is a simple sub-\(\Lambda\)-set. Furthermore, if \(\Pi_k
  \restriction A\) has finite fibers then 
  \(\Pi_k(A)\) is also a simple sub-\(\Lambda\)-set. Finally, if
  \(s \colon \{1, \dots, m\} \to \{1, \dots, k\}\) is a surjective
  map, then
  \[\{(x_1, \dots, x_k) \in \real^k \for (x_{s(1)}, \dots, x_{s(m)})
    \in A\}\]
  is also a simple sub-\(\Lambda\)-set.  \qed
\end{lem}

Throughout the rest of this section, we will be assuming the
following.

\begin{assumption}
  \label{graph-simple}
  For every \(f \in \functionalg m n r\), the graph \(\Gamma(f)\) is a
  simple sub-\(\Lambda\)-set. 
\end{assumption}

\begin{remark}
  All of the hypotheses stated in the previous section were local
  while the assumption above is global. It is meant as a replacement
  for \(\mathcal{A}\)-analyticity (see \cite[Definition
  1.10]{rolin-servi-monomialization}). In particular, notice that if
  \(f \in \functionalg m n r\) is \(\mathcal{A}\)-analytic, then its
  graph \(\Gamma(f)\) is a \(\Lambda\)-set and thus also a simple
  sub-\(\Lambda\)-set. This shows that the current setting is more
  general than that presented in
  \cite{rolin-servi-monomialization}. For instance, the o-minimal
  structure \(\real_H\) presented in \cite{legal-rolin-not-c-infty}
  does not fit the setting of \cite{rolin-servi-monomialization} since
  \(H\) is not \(\mathcal{A}\)-analytic but, as we showed in 
  \cite{guenet-weakly-smooth},
  it verifies assumption \ref{graph-simple}. 
  It is also worth noting that, as an immediate consequence of the
  above assumption, if \(A \subset \real^m\) is a \(\mathcal{A}\)-basic
  set then \(A\) is also a simple sub-\(\Lambda\)-set. 
\end{remark}

\subsection{Parametrizations}
\label{par:parametrizations}

In order to prove that every sub-\(\Lambda\)-set is simple, we need to
obtain a good understanding of their geometry. In view of the
definitions given in Paragraph \ref{par:basic-defs}, it is clear that a first
step towards this goal is to understand the geometry of
\(\mathcal{A}\)-basic sets. This paragraph is dedicated to achieving
this. The main tool that we use to study this question is the
following notion of parametrization. 

\begin{defn}
  \label{defn:parametrization}
  Consider a set \(A \subset \real^m\) and a family \(\{(\rho_j, Q_j)
  \for j \in J\}\) where \(\rho_j \colon \genpolydisk m {r_j} \to
  \real^m\) is an admissible transformation and \(Q_j \subset
  \genpolydisk m {r_j}\) is a sub-quadrant. We say that it is a
  \textbf{local parametrization of \(A\) at \(0\)} when the following
  conditions are satisfied.
  \begin{itemize}
  \item For each \(j \in J\), the function \(\rho_j \restriction Q_j\)
    is a diffeomorphism onto its image \(\rho_j(Q_j) \subset A\). 
  \item If \(U_j \subset \real^m\) is a neighborhood of \(0\) for each
    \(j \in J\) then there is a finite subset \(J_0 \subset J\) such
    that \(\bigcup_{j \in J_0} \rho_j(Q_j \cap U_j)\) is a
    neighborhood of \(0\) in \(A\). 
  \end{itemize}
  If also \(f_1, \dots, f_q\) are real valued functions whose domains
  contain \(A\) then we say that this parametrization is
  \textbf{compatible} with
  the functions \(f_1, \dots, f_q\) when the functions \(f_k \circ
  \rho_j\) have constant sign for each \(j \in J\) and each \(1 \leq k
  \leq q\). 
\end{defn}

\begin{remark}
  The second point in the definition above is meant as some kind of
  compactness property for parametrizations. Accordingly, when we
  construct parametrizations in Lemma \ref{lem:first-parametrization},
  this property is obtained by using compactness of a suitable
  topological space. 
\end{remark}

\begin{lem}
  \label{lem:first-parametrization}
  Consider a polydisk \(\genpolydisk m r \subset \real^m\) and functions
  \(f_1, \dots, f_q \in \genfunctionalg m r\). Then, there exists a local
  parametrization at \(0\) of \(\genpolydisk m r\) which is compatible
  with the functions \(f_1, \dots, f_q\). 
\end{lem}

Notice that the lemma above is obvious when the functions \(f_1,
\dots, f_q\) are normal. Indeed, in this case, the functions \(f_1,
\dots, f_q\) have constant sign on each subquadrant \(Q \subset
\genpolydisk m r\). In the general case, the idea of the proof is to
use a tree \(T\) that \(*\)-monomializes the functions \(f_1, \dots,
f_q\) in order to reduce to the case that \(f_1, \dots, f_q\) are all
normal. 

\begin{proof}
  By \ref{*-monomialization}, there is a tree \(T\) that
  \(*\)-monomializes the functions \(f_1, \dots, f_q\). Up to adding
  suitable reflections at the leaves of the tree, we may also assume
  that each \(\rho\) induced by a branch of the tree is defined on a
  polydisk of the form \(\genpolydisk m {r_\rho}\). 
  The proof now proceeds by induction
  on the pairs \({(m, h)}\) oredered lexicographically where \(h\) is
  the height of \(T\). Firstly, if \(m \in \{0, 1\}\) or if \(h =
  0\) then the germs \(f_1, \dots, f_q\) are
  already normal and the
  result follows at once since normal germs have constant sign on each
  sub-quadrant up to restricting to a neighborhood of \(0\). 

  Now, assume that \(m > 1\) and that \(h > 0\). To begin with, we
  are going to consider the case when the elementary transformations
  attached to the root of \(T\) are not blow-ups.
  In this case, let \(\nu \colon \polydisk {m_\nu} {n_\nu}
  {r_\nu} \to \genpolydisk m r\) be one of these elementary
  transformations. The sub-tree \(T'\) below \(\nu\)
  \(*\)-monomializes the functions \(f_1 \circ \nu, \dots, f_q \circ
  \nu\). Thus, by applying the inductive hypothesis, we get a family
  \(\{(Q^\nu_j, \rho^\nu_j) \for j \in J^\nu\}\) adapted to the
  functions \(f_1 \circ \nu, \dots, f_q \circ \nu\). For each \(j \in
  J^\nu\), the admissible transformation \(\nu \circ \rho_j^\nu\) is
  then clearly a diffeomorphism onto its image and \(f_1 \circ \nu
  \circ \rho^\nu_j, \dots, f_q \circ \nu \circ \rho_j^\nu\) have
  constant sign on \(Q^\nu_j\) by assumption. Finally, consider
  \(U_j^\nu \subset
  \real^{m}\) a neighborhood of \(0\) for each \(\nu\) attached to
  the root of \(T\) and each \(j \in J^\nu\). Then, given \(\nu\),
  there is a finite subset \(J^{\nu}_0 \subset J^\nu\) such that
  \(\bigcup_{j \in J_0^\nu} \rho_j^\nu(Q^\nu_j \cap U^\nu_j)\) is a
  neighborhood of \(0\) in \(\polydisk{m_\nu}{n_\nu}{r_\nu}\). Since
  there are only finitely many elementary transformations attached to
  the root in this case, we are done. \\

  Next, assume that the elementary transformations attached to the
  root are blow-ups and let \(w, w'\) be the two variables that are
  involved in this family of blow-ups. Let \(K\) be the set of
  elementary transformations attached to the root of \(T\) and
  consider \(\nu \colon \polydisk{m_\nu}{n_\nu}{r_\nu}
  \to \genpolydisk{m}{r}\) one of them.
  Let also
  \(w_\nu\) be the critical variable of \(\nu\). Then,
  the sub-tree \(T'\) of \(T\) below \(\nu\)
  \(*\)-monomializes the functions \(f_1 \circ \nu,
  \dots, f_q \circ \nu, w_\nu\), where \(w_\nu\) is viewed as a
  coordinate function. By the inductive hypothesis, we get a
  family \(\{(Q_j^\nu, \rho^\nu_j) \for j \in J^\nu\}\) that is
  adapted to the functions \(f_1 \circ \nu, \dots, f_q \circ \nu,
  w_\nu\). 

  Now, consider \(j \in J^\nu\). Then, the functions \(f_1 \circ \nu
  \circ \rho_j^\nu, \dots, f_q \circ \nu \circ \rho^\nu_j\) have
  constant sign on \(Q_j^\nu\) by assumption. Furthermore, \(w_\nu
  \circ \rho_j^\nu\) also has constant sign on \(Q_j^\nu\). Thus, if
  \(w_\nu \circ \rho_j^\nu \restriction Q_j^\nu \not \equiv 0\), then
  \(w_\nu \circ \rho_j^\nu\) does not vanish on \(Q_j^\nu\) so that
  \(\nu \circ \rho_j^\nu\) is a diffeomorphism onto its image. Thus,
  define
  \[\underline J^\nu \coloneq \{j \in J^\nu \for w_\nu \circ \rho_j^\nu
    \text{ does not vanish on } Q_j^\nu\}.\]
  Now,
  consider \(U_j^\nu \subset \real^{m}\) a neighborhood of \(0\) for
  each \(\nu\) attached to the root of \(T\) and each \(j \in J^\nu\)
  such that \(\nu \circ \rho_j^\nu\) does not vanish on
  \(Q^\nu_j\). If \(j \in J^\nu\) is such that \(\nu \circ \rho_j^\nu
  \restriction Q_j^\nu \equiv 0\), then define \(U_j^\nu = \real^{m}\). For
  each \(\nu\) attached to the root of \(T\), there is a finite
  subset \(J_0^\nu \subset J^\nu\)
  such that \(\bigcup_{j \in J^\nu_0} \rho_j^\nu(Q_j^\nu \cap
  U_j^\nu)\) is a neighborhood of \(0\) in
  \(\polydisk{m_\nu}{n_\nu}{r_\nu}\). Then, by \cite[Remark
  2.17]{rolin-servi-monomialization}, there is a
  finite subset \(K_0 \subset K\) such that
  \[V \coloneq\bigcup_{\nu \in K_0} \bigcup_{j \in J_0^\nu} \nu \circ
    \rho_j^\nu(Q_j^\nu \cap U_j^\nu)\]
  is a neighborhood of \(0\) in \(\polydisk m n r\). Consider now \(a
  \in V\) such that \(w(a) \neq 0\) or \(w'(a) \neq 0\). Then, let
  \(\nu \in K_0\), \(j \in J^\nu_0\) and \(b \in Q_j^\nu \cap U_j^\nu\)
  such that
  \(\nu \circ \rho_j^\nu(b) = a\). We must have \(w_\nu \circ
  \rho_j^\nu(b) \neq 0\) so that \(w_\nu \circ \rho_j^\nu\) does not
  vanish on \(Q_j^\nu\). Thus, letting
  \[\underline J_0^\nu \coloneq \{j \in J_0^\nu \for w_\nu \circ \rho_j^\nu \text{ does
      not vanish on } Q_j^\nu\},\]
  we have that \(\underline J_0^\nu \subset \underline J^\nu\) is
  finite and
  \[\bigcup_{\nu \in K_0} \bigcup_{j \in \underline J_0^\nu} \nu \circ
    \rho_j^\nu(Q_j^\nu \cap U_j^\nu)\]
  is a neighborhood of \(0\) in \(\{a \in \polydisk m r \for w(a)
  \neq 0 \text{ or } w'(a) \neq 0\}\). 
  
  Now, write \(\widehat x\) for the tuple \(x\)
  with the variables \(w\) and \(w'\) removed. By induction on \(m\),
  there is a family \(\{(Q_j', \rho_j') \for j \in J\}\) that is
  adapted to the functions \(f_1(\widehat x, 0, 0), \dots,
  f_q(\widehat x, 0, 0)\). Given \(j \in J\), define
  \[Q_j \coloneq \{x \in \real^{m} \for \widehat x \in Q_j',\  w = 0,\
    w' = 0\}\]
  and let \(\rho_j\) be the extension of
  \(\rho'_j\) to the variables \(x\) obtained by making
  \(\rho_j\) act as the identity on \(w\) and \(w'\). Then, for each
  \(j \in J\), \(\rho_j\) is an admissible transformation such that
  \(\rho_j \restriction Q_j\) is a diffeomorphism onto its
  image. Furthermore, the functions \(f_1 \circ \rho_j, \dots, f_q
  \circ \rho_j\) have constant sign on \(Q_j\). Finally, if \(U_j
  \subset \real^{m}\) is a neighborhood of \(0\) for each \(j \in
  J\), then there is a finite subset \(J_0 \subset J\) such that
  \(\bigcup_{j \in J_0} \rho_j(Q_j \cap U_j)\) is a neighborhood of
  \(0\) in \(\{a \in \polydisk m n r \for w(a) = w'(a) = 0\}\). All in
  all, this shows that the family
  \[\{(Q_j, \rho_j) \for j \in J\} \cup \{(Q_j^\nu, \nu \circ
    \rho_j^\nu) \for \nu \in K,\ j \in \underline J^\nu\}\]
  is adapted to the functions \(f_1, \dots, f_q\) whence the result. 
\end{proof}

\begin{remark}
  The lemma above is inspired by \cite[Proposition
  3.4]{rolin-servi-monomialization} but there are two main
  differences. To begin with, we have to use \(*\)-monomialization
  instead of just using monomialization when handling blow-ups. This
  is necessary to ensure that each chart in the local parametrization
  is a diffeomorphism onto its image. Indeed, in order to write the
  inverse of a blow-up chart \(\nu\), we need to stay away from the
  hyperplane \(\{w_\nu = 0\}\) which can only be done if we know the
  sign of \(w_\nu\) after monomialization. 

  Notice also that we handle compactness quite differently. Indeed, in
  \cite{rolin-servi-monomialization}, the authors use compactness to
  obtain a finite covering which is weaker than what we have proven
  above. Indeed, if we restrict the domains of the charts to smaller
  neighborhoods of \(0\), there is no way to ensure that the charts
  still cover a neighborhood of \(0\) in \(\genpolydisk m r\). In
  \cite{rolin-servi-monomialization}, this strategy is sufficient
  because the authors do not need to restrict the domains of the chart
  at any point. However, we will do so quite often hence the need to
  obtain a ``finiteness property after suitable restrictions of the
  domain'', which is exactly the content of the second point in
  Definition \ref{defn:parametrization}. 
\end{remark}

The lemma above is seemingly weaker than \cite[Proposition
3.4]{rolin-servi-monomialization} because we have constructed local
parametrizations for polydisks only, and not for general
\(\mathcal{A}\)-basic sets. However, since \(\mathcal{A}\)-basic sets
are defined by equations and inequations, it is easy to refine the
lemma to fix this shortcoming. 

\begin{prp}
  \label{prp:first-parametrization}
  Let \(A \subset \genpolydisk m r\) be a \(\mathcal{A}\)-basic set
  and consider functions
  \(f_1, \dots, f_q \in \genfunctionalg m r\). Then, there exists a
  local parametrization of \(A\) at \(0\) which is adapted to the
  functions \(f_1, \dots, f_q\). 
\end{prp}

\begin{proof}
  Up to shrinking \(\genpolydisk m r\), we may assume that there are \(f,
  g_1, \dots, g_p \in \genfunctionalg m r\) such that
  \[A = \{x \in \genpolydisk m r \for f(x) = 0, g_1(x) > 0, \dots, g_p(x)
    > 0\}.\]
  Now, consider \(\{(\rho_j, Q_j) \for j \in J\}\) a local parametrization
  of \(\genpolydisk m r\) at \(0\) compatible with the functions \(f, f_1, \dots,
  f_q, g_1, \dots, g_p\) and let \(J' = \{j \in J \for \rho_j(Q_j)
  \subset A\}\). The family \(\{(\rho_j, Q_j) \for j
  \in J'\}\) is a local parametrization of \(A\) at \(0\) compatible with the
  functions \(f_1, \dots, f_q\).
\end{proof}

Recall that we want to study the geometry of
sub-\(\Lambda\)-sets. Thus, we would like to derive similar results to
the above, but for projections of \(\mathcal{A}\)-basic sets
instead. In order to do so, the idea is to sharpen Proposition
\ref{prp:first-parametrization}
by making each chart in the local parametrization be compatible with a
certain projection. Accordingly, the rest of this paragraph is
dedicated to proving
Corollary \ref{cor:best-parametrization}. The corollary is inspired by
the paragraph before the Fiber Cutting Lemma (Lemma 4.5) in
\cite{rolin-speissegger-wilkie-denjoy-carleman-classes}. The main
difference between the present treatment and the classical one (see
\cite{rolin-speissegger-wilkie-denjoy-carleman-classes} or
\cite{rolin-servi-monomialization} for instance) is that, instead of
decomposing manifolds into smaller submanifolds, we further
parametrize the manifolds obtained in Proposition
\ref{prp:first-parametrization}. To do so, we
need to be able to ``compose'' parametrizations which is the content
of the following lemma. 

\begin{lem}
  \label{lem:parametrization-composition}
  Let \(A \subset \real^m\) be a set and consider \(\{(\rho_j, Q_j)
  \for j \in J\}\) a local parametrization of \(A\) at \(0\). Let also
  \(\{(\rho_{jk}, Q_{jk}) \for k \in K_j\}\) be a local
  parametrization of \(Q_j\) at \(0\) for each \(j \in J\). Then
  \[\{(\rho_j \circ \rho_{jk}, Q_{jk}) \for j \in J, k \in K\}\]
  is a local parametrization of \(A\) at \(0\). \qed
\end{lem}

As in \cite{rolin-speissegger-wilkie-denjoy-carleman-classes}, we are
going to decompose the various manifolds obtained in function of the
sign of some determinants involving a basis for the tangent plane of
the manifolds in question. Such bases can be obtained by
straightfoward computations with matrices. However, we must be careful
to replace every division by a multiplication in order to stay within
the quasianalytic algebras. These issues are explained in more detail
below.

\begin{digr}
  \label{jacobian-algebra}
  Consider \(\eta \colon \genpolydisk d r \to \genpolydisk {m+k} s\) a function
  whose components are all in \(\genfunctionalg d r\). Let \(a \colon
  \genpolydisk d r \to \real\) be the function defined by \(a(x) = x_1
  \dots x_d\), it is clear that \(a \in \genfunctionalg d
  r\). Consider also \(Q \subset \genpolydisk d r\) the open sub-quadrant
  of \(\genpolydisk d r\). Up to shrinking \(\genpolydisk d r\), we may assume that
  \(x_i \frac{\partial \eta_i}{\partial x_j} \in \genfunctionalg d r\)
  for every \(1 \leq i \leq m+k\) and \(1 \leq j \leq d\). Thus, if we
  define the matrix \(A(x) = (a(x)\frac{\partial \eta_i}{\partial
    x_j})_{1 \leq i \leq m+k, 1 \leq j \leq d}\), then all of the
  entries of \(A(x)\) are in \(\genfunctionalg d r\). Furthermore, if
  \(x \in Q\), then \(\eta\) is differentiable at \(x\) and \(A(x) =
  a(x) d_x(\eta)\), where \(a(x) \neq 0\).

  Assume in particular that \(\eta \restriction Q\) is a
  diffeomorphism onto \(M = \eta(Q)\). Then, \(A(x)\) is injective
  for every \(x \in Q\). If \(e_1, \dots, e_d\) is the canonical basis
  of \(\real^d\) and if \(a_i(x) = A(x)e_i\) for each \(x \in
  \genpolydisk d r\) and each \(1 \leq i \leq d\), then all the
  components of \(a_i\) are \(\genfunctionalg d r\) and \(a_1(x),
  \dots, a_d(x)\) is a basis of \(\tangent {\eta(x)} M\). 

  From now on, we also suppose that \(\Pi_m
  \restriction M\) has constant rank \(l\) and that there is a
  strictly increasing sequence \(\iota \colon \{1, \dots, d\} \to \{1,
  \dots, m+k\}\) such that \(\iota(l) \leq m\) and \(\Pi_\iota
  \restriction M\) is an immersion. For \(x \in \genpolydisk d r\), we
  define the matrix \(A'(x) = (a(x) \frac{\partial
    \eta_{\iota(i)}}{x_j})_{1 \leq i
    \leq d, 1 \leq j \leq d}\). All of the entries of \(A'(x)\) are in
  \(\genfunctionalg d r\) and, for \(x \in Q\), we have \(A'(x) = a(x)
  d_x(\Pi_\iota \circ \eta)\) so that \(A'(x)\) is
  invertible. Consider also \(B(x)\) the transpose of the cofactor
  matrix of \(A'(x)\). Then, all of the entries of \(B(x)\) are in
  \(\genfunctionalg d r\) and, for \(x \in Q\), we have \(B(x) =
  \det(A'(x)) A'(x)^{-1}\). Finally, define \(b_i(x) = B(x)e_i\) for \(1
  \leq i \leq d\). It is clear that all of the components of \(b_i\)
  are in \(\genfunctionalg d r\). Fix some \(x \in Q\), we have that
  \(b_1(x), \dots, b_d(x)\) is
  a basis of \(\real^d\). Furthermore, since \(\ker (\Pi_m
  \restriction \tangent {\eta(x)} M)\) has dimension \(d-l\), it
  follows that \(\Pi_\iota\) induces an isomorphism from
  \(\ker(\Pi_m \restriction \tangent {\eta(x)} M)\) to
  \(\ker(\Pi_l)\). If \(1 \leq i \leq d-l\), then
  \[\Pi_l \circ \Pi_\iota \circ d_x(\eta)(b_i(x)) =
    \Pi_l\left(\frac{\det(A'(x))}{a(x)}e_i \right) = 0\]
  so that \(\Pi_m \circ d_x(\eta)(b_i(x)) = 0\). Thus, \(b_{d-l+1}(x),
  \dots, b_d(x)\) are a basis of \(\ker(\Pi_m \circ d_x(\eta))\).

  We are especially interested in the following situation. Consider
  \(Q \subset \genpolydisk {m+k} s\) a sub-quadrant and let \(\rho
  \colon \genpolydisk {m+k} s \to \real^{m+k}\) be an admissible
  transformation such that \(\rho \restriction Q\) is a diffeomorphism
  onto its image. If we write \(d = \dim(Q)\), then there is a unique
  strictly increasing sequence \(\kappa \colon \{1, \dots, d\} \to
  \{1, \dots, m+k\}\) such that \(\Pi_\kappa \restriction Q\) is a
  diffeomorphism onto its image \(Q' \subset \genpolydisk d r\), where
  \(r = \Pi_\kappa(s)\). Then, \(Q'\) is the open sub-quadrant of
  \(\genpolydisk d r\) and we might apply the previous results to the
  map \(\eta \colon \genpolydisk d r \to \real^{m+k}\) defined by
  \(\eta(x) = \rho(y)\) where \(y_i = x_j\) when \(\kappa(j) = i\) and
  \(y_i = 0\) when \(i \not \in \ima(\kappa)\). 
\end{digr}

\begin{prp}
  \label{prp:parametrization-to-compose}
  Let \(Q \subset
  \genpolydisk {m+k} r\) be a sub-quadrant and \(\rho \colon
  \genpolydisk {m+k} r
  \to \real^{m+k}\) be an admissible transformation such that the
  restriction \(\rho \restriction Q\) is a
  diffeomorphism onto its image. There is a local parametrization 
  \(\{(\rho_j, Q_j) \for j \in J\}\) of \(Q\) at \(0\) such that, for
  each \(j \in J\),
  if we write \(M_j = \rho \circ \rho_j(Q_j)\) then
  \begin{itemize}
  \item The projection \(\Pi_m \restriction M_j\) has constant rank
    \(l\). 
  \item There is a strictly increasing sequence \(\iota \colon \{1,
    \dots, d\} \to \{1, \dots, m+k\}\) such that \(\iota(l) \leq m\)
    and \(\Pi_\iota \restriction M_j\) is an immersion. 
  \end{itemize}
\end{prp}

\begin{proof}
  We prove the result by induction on \(\dim(Q)\), the result being
  obvious when \(\dim(Q) = 0\). Thus, assume that \(d = \dim(Q) >
  0\) and write \(M = \rho(Q)\). By \ref{jacobian-algebra}, up to
  shrinking \(\genpolydisk {m+k} r\), we
  might assume that there are functions \(a_1, \dots,
  a_d \colon \genpolydisk {m+k} r \to \real^m\) such that all of their
  components are in \(\genfunctionalg {m+k} r\) and \(a_1(x), \dots,
  a_d(x)\) is a basis of \(\tangent {\rho(x)} M\) for each \(x \in
  Q\). Both points of the proposition are proven by introducing 
  parametrizations that are compatible with some family of
  determinants involving \(a_1, \dots, a_d\). 

  Firstly, consider a local parametrization \(\{(\rho_j, Q_j) \for j \in J\}\)
  of \(Q\) at \(0\) that is compatible with the functions
  \[\det(\Pi_\iota(a_{\kappa(1)}(x)), \dots,
    \Pi_\iota(a_{\kappa(k)}(x)))\]
  for every integer \(k \geq 0\) and every strictly increasing
  functions \(\iota \colon \{1, \dots, k\} \to \{1, \dots, m\}\) and
  \(\kappa \colon \{1, \dots, k\} \to \{1, \dots, d\}\). By Lemma
  \ref{lem:parametrization-composition}, it suffices to show that the
  result holds for each pair \((\rho \circ \rho_j, Q_j)\). Fix \(j \in
  J\), if
  \(\dim(Q_j) < \dim(Q)\), we are done by induction hypothesis. Thus,
  assume that \(\dim(Q_j) = \dim(Q)\) and let \(M_j = \rho \circ
  \rho_j(Q_j)\). Since \(\dim(M_j) = \dim(M)\), it follows that
  \(M_j\) is an open submanifold of \(M\). Thus, for every \(x \in
  Q_j\), the family \(a_1 \circ \rho_j(x), \dots, a_d \circ
  \rho_j(x)\) is a basis of \(\tangent {\rho \circ \rho_j(x)}
  {M_j}\). Furthermore, for every integer \(k \geq 0\) and every
  increasing sequences \(\iota \colon \{1, \dots, k\} \to \{1, \dots,
  m\}\) and \(\kappa \colon \{1, \dots, k\}\to \{1, \dots, d\}\), the
  function
  \[\det(\Pi_\iota(a_{\kappa(1)} \circ \rho_j(x)), \dots,
    \Pi_\iota(a_{\kappa(k)} \circ \rho_j(x)))\]
  has constant sign. Thus, it follows at once that \(\Pi_m
  \restriction M_j\) has constant rank \(l\) for some integer \(l \geq
  0\). All in all, we have shown that
  it suffices to prove the proposition when the first point is already
  satisfied.

  Thus, assume that the first point holds and consider \(\{(\rho_j,
  Q_j) \for j \in J\}\) a local parametrization of \(Q\) at \(0\) that
  is compatible with the functions
  \[\det(\Pi_\iota(a_1(x)), \dots, \Pi_\iota(a_d(x)))\]
  for every strictly increasing sequence \(\iota \colon \{1, \dots,
  d\} \to \{1, \dots, m+k\}\) such that \(\iota(l) \leq m\). By Lemma
  \ref{lem:parametrization-composition}, it suffices to prove the result
  for each pair \((\rho \circ \rho_j, Q_j)\). Fix \(j \in J\), if
  \(\dim(Q_j) < \dim(Q)\) then the result follows by the induction
  hypothesis. Thus, assume that \(\dim(Q_j) = \dim(Q)\) so that \(M_j
  \coloneq \rho \circ \rho_j(Q_j)\) is an open submanifold of
  \(M\). Then, \(\Pi_m \restriction M_j\) must have constant rank
  \(l\). Let \(x \in Q_j\) and write \(y = \rho_j(x)\).
  Then the family \(\Pi_m(a_1(y)), \dots,
  \Pi_m(a_d(y))\) has rank \(l\) so that, up to changing the order, we
  might assume that \(\Pi_m(a_1(y)), \dots, \Pi_m(a_l(y))\) are
  independent. Thus, there is some strictly increasing sequence
  \(\iota' \colon \{1, \dots, l\} \to \{1, \dots, m\}\) such that
  \(\Pi_{\iota'}(a_1(y)), \dots, \Pi_{\iota'}(a_l(y))\) is a basis of
  \(\real^l\). Since the vectors \(a_1(y), \dots, a_d(y)\) are
  independent, there exists a strictly increasing sequence \(\iota
  \colon \{1, \dots, d\} \to \{1, \dots, m+k\}\) which extends \(\iota'\)
  and such that \(\Pi_\iota(a_1(y)), \dots, \Pi_\iota(a_d(y))\) is a
  basis of \(\real^d\). In particular, \(\iota(l) \leq m\) and
  \[\det(\Pi_\iota(a_1 \circ \rho_j(x)), \dots, \Pi_\iota(a_d \circ \rho_j(x)))
    \neq 0.\]
  Since this is a constant function of \(x\), it follows that
  \(\Pi_\iota \restriction M_j\) has to be an immersion whence the
  result. 
\end{proof}

\begin{remark}
  Notice that the induction on dimension is necessary in the proof
  above. Indeed, reusing the notations of the second paragraph, if
  \(\dim(Q_j) < \dim(Q)\) then we do not know a basis of \(\tangent
  {\rho \circ \rho_j(x)} {M_j}\) for \(x \in Q_j\). In particular, the
  determinants
  \[\det(\Pi_\iota(a_{\kappa(1)} \circ \rho_j(x), \dots, a_{\kappa(k)}
    \circ \rho_j(x)))\]
  have no meaning in this context.
\end{remark}

The following corollary is the result we have been building up to in
this section. As explained earlier, it gives parametrizations for
\(\mathcal{A}\)-basic sets that are ``compatible'' with a certain
projection. We intend to use the corollary in conjunction with the
Fiber Cutting Lemma to be introduced in the next paragraph to obtain
results about the geometry of sub-\(\Lambda\)-sets.

\begin{cor}
  \label{cor:best-parametrization}
  Let \(A \subset \real^{m+k}\) be a \(\mathcal{A}\)-basic set. Then,
  there exists a local parametrization \(\{(\rho_j, Q_j) \for j \in J\}\) of
  \(A\) at \(0\) such that, for every \(j \in J\), we have
  \begin{itemize}
  \item If we write \(M_j = \rho_j(Q_j)\) then \(\Pi_m \restriction
    M_j\) has constant rank \(l\). 
  \item There is a strictly increasing sequence \(\iota \colon \{1,
    \dots, d\} \to \{1, \dots, m+k\}\) such that \(\iota(l) \leq m\)
    and \(\Pi_\iota \restriction M_j\) is an immersion. 
  \end{itemize}
\end{cor}

\begin{proof}
  We may first obtain a local parametrization \(\{(\rho_j, Q_j) \for
  j \in J\}\) of \(A\) at \(0\) by Proposition
  \ref{prp:first-parametrization}. For each \(j \in J\), we consider
  \(\{(\rho_{jk}, Q_{jk}) \for k \in K_j\}\) the local parametrization
  of \(Q\) at \(0\) obtained by Proposition
  \ref{prp:parametrization-to-compose}. Finally, by Lemma
  \ref{lem:parametrization-composition}, the family \(\{(\rho_j \circ \rho_{jk},
  Q_{jk}) \for j \in J, k \in K_j\}\) is a local parametrization of
  \(A\) at \(0\) whence the result. 
\end{proof}

\begin{remark}
  \label{rem:hypotheses-fiber-cutting}
  Consider \(Q \subset \genpolydisk {m+k} r\) a sub-quadrant of
  dimension \(d\) and \(\rho \colon \genpolydisk {m+k} r \to
  \real^{m+k}\) an admissible transformation such that \(\rho
  \restriction Q\) is a diffeomorphism onto its image. Consider also
  \(M = \rho(Q)\) and assume that \(\Pi_m \restriction M\) has
  constant rank \(l < d\) and that there is a strictly increasing
  sequence \(\iota \colon \{1, \dots, d\} \to \{1, \dots, m+k\}\) such
  that \(\Pi_\iota \restriction M\) is an immersion and \(\iota(l)
  \leq m\). Let \(f = \rho_{\iota(l+1)}\) and consider \(x \in Q\). We
  are going to show that \(\grad f(x)\) is not in the orthogonal of
  \(\ker(\Pi_m \circ d_x(\rho))\).

  Indeed, let \(e_1, \dots, e_{d-l}\)
  be a basis of \(\ker(\Pi_m \circ d_x(\rho))\). The family of vectors
  \(\Pi_\iota(\tangent x \rho(e_1)), \dots, \Pi_\iota(\tangent x \rho(e_{d-l}))\) are linearly
  independent and we have \(\Pi_m(\tangent x \rho(e_1)) = \dots =
  \Pi_m(\tangent x \rho(e_{d-l})) = 0\). Thus, if we
  define \(\iota' \colon \{1, \dots, d-l\} \to \{1, \dots, m+k\}\) by
  setting \(\iota'(i) = \iota(l+i)\), the vectors
  \(\Pi_{\iota'}(\tangent x \rho (e_1)),
  \dots, \Pi_{\iota'}(\tangent x \rho (e_{d-l}))\) are independent so that they are a
  basis of \(\real^{d-l}\). In particular, there must be some \(1 \leq i
  \leq d-l\) such that the \(\iota(l+1)\)-th coordinate of \(\tangent
  x \rho(e_i)\) is non-zero. Thus, \(\grad \rho_{\iota(l+1)} \cdot e_i
  \neq 0\) whence the result.

  Combining what we said above with \ref{jacobian-algebra}, it is easy
  to show that the charts obtained by Corollary
  \ref{cor:best-parametrization} verify the assumptions of the Local
  Fiber Cutting Lemma (Lemma \ref{fiber-cutting}). 
\end{remark}

Before ending this paragraph, we prove a few results about the
geometry of \(\Lambda\)-sets and sub-\(\Lambda\)-sets that are easy
consequences of the corollary above. To begin with, if \(A \subset
\real^{m+k}\) is a \(\Lambda\)-set, we can use the compactness of
\(\clos A\) to obtain the following:

\begin{cor}
  \label{global-parametrization}
  Let \(A \subset \real^{m+k}\) be a \(\Lambda\)-set. Then, there are
  \(M_1, \dots, M_p\) some simple sub-\(\Lambda\)-sets that are also
  connected manifolds such that \(A = M_1 \cup \dots \cup M_p\), and, for each
  \(1 \leq i \leq p\), the projection \(\Pi_m \restriction M_i\) has
  constant rank \(l_i\) and there is a strictly increasing sequence
  \(\iota \colon \{1, \dots, d_i\} \to \{1, \dots, m+k\}\), where
  \(d_i = \dim(M_i)\), such that \(\Pi_\iota \restriction M_i\) is an
  immersion and \(\iota(l_i) \leq m\). \qed
\end{cor}

\begin{remark}
  Notice that, as an immediate consequence of the corollary,
  \(\Lambda\)-sets have finitely many connected components. Since
  sub-\(\Lambda\)-sets are continuous images of \(\Lambda\)-sets, they
  too have only finitely many connected components. 
\end{remark}

Since the image of a manifold by a function with constant rank has
dimension, we can use the corollary above to show that
sub-\(\Lambda\)-sets have dimension. The next lemma will allow
us to compute that dimension in some key cases. It shows in particular
that, if \(A \subset \real^{m+k}\) is a \(\Lambda\)-set and \(\Pi_m
\restriction A\) has finite fibers, then \(\dim(\Pi_m(A)) = \dim(A)\)
which is useful when considering simple sub-\(\Lambda\)-sets.  

\begin{lem}
  \label{dimension-fibers}
  Consider \(A \subset \real^{m+k}\) a simple sub-\(\Lambda\)-set and
  assume that there is an
  integer \(\mu \geq 0\) such that \(\dim A_y = \mu\) for all \(y \in
  \Pi_m(A)\). Then \(\dim (A) = \dim (\Pi_m(A)) + \mu\). 
\end{lem}

\begin{proof}
  We are first going to prove the result when \(A\) is a
  \(\Lambda\)-set. To this end, consider \(M_1, \dots, M_p\) as in
  Corollary \ref{global-parametrization}. Let \(d_i = \dim(M_i)\) and
  \(l_i\) be the constant rank of \(\Pi_m \restriction M_i\). If we
  define \(\mu_i = d_i - l_i\), then \(\dim((M_i)_y) = \mu_i\) for each
  \(y \in \Pi_m(M_i)\). Since \((M_i)_y \subset A_y\),
  it follows that \(\mu_i \leq \mu\). Furthermore, \(\dim(\Pi_m(M_i)) =
  l_i\) so that \(l_i \leq \dim((\Pi_m(A))\). Thus, \(d_i \leq
  \dim(\Pi_m(A)) + \mu\). But \(\dim(A) = \max(d_1, \dots, d_p)\) so
  that \(\dim(A) \leq \dim(\Pi_m(A)) + \mu\).

  Now, we prove the opposite inequality. To this end, consider \(y \in
  \Pi_m(A)\), we have \(A_y = (M_1)_y \cup \dots \cup (M_p)_y\). Since
  \(\dim(A_y) = \mu\), there is some \(1 \leq i \leq p\) such that
  \(\dim((M_i)_y) = \mu\). But then, \(\mu_i = \mu\) and \(y \in
  \Pi_m(M_i)\) so that we obtain
  \[\Pi_m(A) = \bigcup_{\mu_i = \mu} \Pi_m(M_i).\]
  Thus, there is some \(1 \leq i \leq p\) such that \(\mu_i = \mu\) and
  such that \(l_i = \dim(\Pi_m(M_i)) = \dim(\Pi_m(A))\). We now have
  \(\dim(A) \geq \dim(M_i) = l_i+\mu_i = \dim(\Pi_m(A)) + \mu\) whence the
  result when \(A\) is a \(\Lambda\)-set. 
  
  Assume next that \(A\) is a simple sub-\(\Lambda\)-set and consider
  an integer \(k \geq 0\) and a \(\Lambda\)-set \(A' \subset
  \real^{m+k+h}\) such that \(\Pi_{m+k}(A') = A\) and \(\Pi_{m+k}
  \restriction A'\) has finite fibers. In particular, we may apply
  the lemma to \(A'\) to obtain that \(\dim(A') =
  \dim(A)\). Furthermore, \(\Pi_m(A) = \Pi_m(A')\) and, for every \(y
  \in \Pi_m(A')\), we have \(\Pi_{m+k}(A'_y) = A_y\). Since
  \(A'_y\) is a \(\Lambda\)-set and \(\Pi_{m+k} \restriction A'_y\)
  has finite fibers, we conclude that \(\dim(A'_y) = \dim(A_y) = \mu\)
  by the lemma. Finally, applying the lemma again, we find that
  \(\dim(A) = \dim(A') = \dim(\Pi_m(A')) + \mu = \dim(\Pi_m(A)) + \mu\)
  which concludes the proof. 
\end{proof}

\subsection{The Fiber Cutting Lemmas}
\label{par:fiber-cutting}

For every sub-\(\Lambda\)-set \(A \subset \real^n\), there exist some
integer \(k \geq 0\) and some \(\Lambda\)-set \(A' \subset
\real^{n+k}\) such that \(\Pi_{n+k}(A') = A\). However, it is possible
that \(\dim A' > \dim A\). Thanks to the Global Fiber Cutting
Lemma (stated as Corollary \ref{global-fiber-cutting} below), we are
able to replace \(A'\) with \(\Lambda\)-sets of lower dimension
whenever this is the case. It is then rather easy to deduce Theorem
\ref{sub-sets-are-simple} according to which all sub-\(\Lambda\)-sets
are simple. We first prove a local version of the Fiber Cutting Lemma,
from which the global version follows by compactness of
parametrizations.

\begin{lem}[Local Fiber Cutting Lemma]
  \label{fiber-cutting}
  Let \(Q \subset \genpolydisk d r\) be a sub-quadrant of \(\genpolydisk d
  r\). Consider also \(\eta \colon \genpolydisk d r \to \real^m\) a
  \(C^1\)-map such that \(\eta \restriction Q\) has constant rank \(l
  < d\) and \(\eta_i \in
  \genfunctionalg d r\) for each \(1 \leq i \leq m\) where \(\eta =
  (\eta_1, \dots, \eta_m)\). Assume also that
  \begin{itemize}
  \item There are functions \(b_1, \dots, b_{d-l} \colon \genpolydisk d r
    \to \real^{m+k}\) such that all of their components are in
    \(\genfunctionalg d r\) and, for
    every \(x \in Q\), \(b_1(x), \dots, b_{d-l}(x)\) is a
    basis of \(\ker(\tangent x \eta)\). 
  \item There is a function \(f \in \genfunctionalg d r\) such that,
    for all \(x \in Q\),
    \(\grad f(x)\) is not in the orthogonal of \(\ker(\tangent x
    \eta)\). 
  \end{itemize}
  Then, there exists a simple sub-\(\Lambda\)-set \(A \subset Q\) and a
  neighborhood \(U \subset \real^d\) of \(0\) such that \(\eta(Q \cap
  U) = \eta(A)\) and \(\dim(A) < d\). 
\end{lem}

\begin{remark}
  The function \(f\) whose existence is assumed will only be useful at
  the end of the proof. It serves as a way to show that the connected
  components of the fibers of \(\eta \restriction Q\) must have
  non-empty frontier following the discussion before Lemma 4.5 in
  \cite{rolin-speissegger-wilkie-denjoy-carleman-classes}. 
\end{remark}

\begin{proof}
  Notice first that it suffices to prove the result when \(Q\) is an
  open sub-quadrant of \(\genpolydisk d r\). Then, define
  \[\widetilde Q = \{(x, r') \in \real^d \times \real^d \for 0 < r'_i < r_i
    \text{ for all } 1 \leq i \leq d \text{ and } x \in Q \cap
    \genpolydisk d {r'}\}\]
  and consider the function \(\Phi \colon \real^d \times \real^d \to
  \real\) defined by
  \[\Phi(x, r') = \prod_{i=1}^d x_i(r'_i - x_i).\]
  Let also
  \[\widetilde A = \{(x, r') \in \widetilde Q \for \grad_x \Phi(x, r') \cdot b_1(x) = \dots
    = \grad_x \Phi(x, r') \cdot b_{d-l}(x) = 0\}.\]
  Now, fix some \(r'\) such that \(0 < r'_i < r_i\) for every \(1 \leq
  i \leq d\). The first step is to show that \(\eta(\widetilde A_{r'})
  = \eta(Q \cap \genpolydisk d {r'})\). To do this, we proceed as in
  \cite[Lemma 3.11]{rolin-servi-monomialization}. Thus
  let \(y \in \eta(Q \cap \genpolydisk d {r'})\) and consider
  \(C\) a connected component of the fiber of \(\eta \restriction Q
  \cap \genpolydisk d {r'}\) above \(y\). 
  Since \(\clos C\) is compact, there is \(x_0 \in \clos C\) such that
  \(\Phi(x_0, r')\) is
  maximal. Given that \(x \mapsto \Phi(x, r')\) is stricly positive on \(Q
  \cap \genpolydisk d {r'}\) and vanishes on its frontier, it follows that
  \(x_0 \in C\). Thus, the restriction of \(x \mapsto \Phi(x, r')\) to
  \(C\) is critical at \(x_0\) so that \(x_0 \in \widetilde A_{r'}\). This
  shows in particular that \(\eta(\widetilde A_{r'}) = \eta(Q \cap
  \genpolydisk d {r'})\) as announced above. Furthermore, \(\widetilde
  A_{r'}\) is a simple
  sub-\(\Lambda\)-set so that we might pick \(A = \Pi_d(\widetilde A_{r'})\)
  whenever \(\dim (\widetilde A_{r'}) < d\).

  In \cite[Lemma 3.11]{rolin-servi-monomialization}, the authors show
  that we must have \(\dim(\widetilde A_{r'}) < d\) for every
  \(r'\). However, this involves using \(\mathcal{A}\)-analyticity
  which we have not assumed. Thus, we will instead attempt to show
  that there is some \(r'\) such that \(\dim(\widetilde A_{r'}) <
  d\). 
  \begin{figure}
    \begin{tikzpicture}[scale = 0.6]
      \draw[step=1cm,gray,very thin,opacity=0.5] (-4.9,-4.9) grid
      (4.9,4.9);
      \draw[thick,->] (-4.5,0) -- (4.5,0) node[anchor =
      north east]{\(r'\)};
      \draw[thick,->] (0,-4.5) -- (0,4.5)
      node[anchor = north east]{\(x\)}; \draw (-4, -4) -- (4, 4);
      \draw (-4, 4) -- (4, -4);
      \fill [orange!60!white, opacity = 0.3] (-4, 4) -- (-4, -4) --
      (0, 0) -- cycle; \fill [orange!60!white, opacity = 0.3] (4, 4)
      -- (4, -4) -- (0, 0) -- cycle;
    \end{tikzpicture}
    \hspace{20pt}
    \begin{tikzpicture}[scale=0.6]
      \draw[step=1cm,gray,very thin,opacity=0.5] (-4.9,-4.9) grid
      (4.9,4.9);
      \draw[thick,->] (-4.5,0) -- (4.5,0) node[anchor = north east]{\(r'\)};
      \draw[thick,->] (0,-4.5) -- (0,4.5) node[anchor = north east]{\(x\)};
      \draw (-4, -4) -- (4, 4);
      \draw (-4, 4) -- (4,-4);
      \draw (-3, -2) rectangle (3, 2) node[anchor=north west]
      {\(\genpolydisk {2d} s\)};
      \fill [green!60!white, opacity = 0.3] (-2, -2) rectangle (2, 2);
    \end{tikzpicture}
    \caption{}
    \label{fig:fiber-cutting}
  \end{figure}
  The argument will rely on an ability to restrict to arbitrary
  neighborhoods of \(0\) in \(\widetilde Q\). More precisely, given a
  polydisk \(\genpolydisk {2d} s\), we have \(\widetilde A_{r'}
  \subset \genpolydisk {2d} s\) for every sufficiently small
  \(r'\). This is illustrated in the case \(d=1\), \(r=4\) and \(s =
  (2, 3)\) in
  Figure \ref{fig:fiber-cutting}. Indeed, the set \(\widetilde Q\) is
  included in the orange area in the leftmost drawing whence the
  fibers in the green area in the rightmost drawing are fully
  contained in \(\genpolydisk {2d} s\). 
  
  Thus, assume by contradiction that \(\dim(\widetilde A_{r'}) = d\) for every
  \(r'\) such that \(0 < r'_i < r_i\) for each \(1 \leq i \leq
  d\) and let \(B\) be the interior of \(\widetilde A\).
  We are going to show that \(0 \in \clos B\). 
  To this end, consider \(\genpolydisk {2d} s \subset \real^{2d}\) a
  polydisk. For every sufficiently small \(r'\), we have \(\widetilde A_{r'}
  \subset \genpolydisk {2d} s\) so that, by Lemma \ref{dimension-fibers}, it follows that
  \(\dim(\widetilde A\cap \genpolydisk {2d} s) = 2d\). Thus, \(\widetilde A\cap
  \genpolydisk {2d} s\)
  has non-empty interior so that \(B \cap \genpolydisk {2d} s \neq
  \varnothing\). This shows that \(0 \in \clos B\) as we wanted. Consider a
  parametrization \(\{(\rho_j, Q_j) \for j \in J\}\) of \(\genpolydisk d
  r \times \genpolydisk d r\) that is compatible with the functions
  \(a_1, \dots, a_{d-l} \in \genfunctionalg {2d} {(r, r)}\), where
  \[a_i(x, r') = \grad_x \Phi(x, r') \cdot b_i(x),\]
  and let \(J_0 \subset J\) be a finite subset such that
  \(\bigcup_{j \in J_0} \rho_j(Q_j)\) is a neighborhood of \(0\) in
  \(\genpolydisk d r \times \genpolydisk d r\). Since \(B\) is an open set
  and \(0 \in \clos B\), there has to be some \(j \in J_0\) such that
  \(\dim(B \cap \rho_j(Q_j)) = 2d\). It follows that \(Q_j\) is an
  open sub-quadrant and that \(a_i \circ \rho_j \restriction Q_j
  \equiv 0\) for each \(1 \leq i \leq d-l\). Thus, for \(1 \leq i
  \leq d-l\), we have \(\taylor (a_i \circ \rho_j) = 0\) so that
  \(\taylor (a_i) = 0\). But then, by quasi-analyticity, we deduce
  that there is a polydisk \(\genpolydisk {2d} s \subset \real^{2d}\) such
  that \(a_1, \dots, a_{d-l}\) vanish on \(\genpolydisk {2d} s\).

  Now, for
  every sufficiently small \(r'\), we have \(\widetilde Q_{r'}\subset
  \genpolydisk {2d}
  s\). Consider such a \(r'\), some \(y \in \eta(Q \cap
  \genpolydisk d {r'})\) and let \(C\) be a connected component of the
  fiber of \(\eta \restriction Q \cap \genpolydisk d {r'}\) above
  \(y\).
  Then, \(\Phi(x, r')\) is constant and
  strictly positive on \(C\) and it vanishes on the frontier of
  \(C\). This implies that \(C\) has empty frontier whence it is
  compact. Thus, there is some \(x_0 \in C\) such that \(f(x_0)\) is
  maximal. But, by hypothesis, there is some \(1 \leq i \leq d-l\)
  such that \(\grad f(x_0) \cdot b_i(x_0) \neq 0\) which contradicts
  the maximality. 
\end{proof}

\begin{remark}
  Reusing the notation of the proof above, the function \(x \mapsto
  \Phi(x, r')\) for some fixed \(r'\) is simpler than the function
  \(g\) in \cite[Lemma 3.11]{rolin-servi-monomialization} since there
  are no functions \(g_1, \dots, g_q\). This is a consequence of our
  approach towards refining the parametrization. Indeed, the functions
  \(g_1, \dots, g_q\) in \cite{rolin-servi-monomialization} are used
  to define an open submanifold. Since we have decided to parametrize
  further instead of taking submanifolds, these functions do not
  appear. 
  
  Assume that \(Q \subset \genpolydisk {m+k} r\) is a sub-quadrant and
  that \(\rho \colon \genpolydisk {m+k} r \to \real^{m+k}\) is an
  admissible transformation such that \(\rho \restriction Q\) is a
  diffeomorphism onto \(M \coloneq \rho(Q)\). Assume also that \(\Pi_m
  \restriction M\) has constant rank \(l < d\) and that there is a
  strictly increasing sequence \(\iota \colon \{1, \dots, d\} \to \{1,
  \dots, m+k\}\) such that \(\iota(l) \leq m\) and \(\Pi_\iota
  \restriction M\) is an immersion. Then, by \ref{jacobian-algebra}
  and by Remark \ref{rem:hypotheses-fiber-cutting}, it
  follows that the function \(\eta \coloneq \Pi_m \circ \rho\)
  satisfies the hypotheses of the lemma. Thus, reusing the notations
  of Corollary \ref{cor:best-parametrization},
  for each \(j \in J\), either we can apply the Local Fiber Cutting
  Lemma to \(\eta_j \coloneq \Pi_m \circ \rho_j\) over the
  sub-quadrant \(Q_j\) or the projection \(\Pi_m \restriction M_j\)
  has finite fibers. 
\end{remark}

By the remark above, we can apply the Local Fiber Cutting Lemma to the
charts obtained thanks to Corollary
\ref{cor:best-parametrization}. The proposition below is obtained by
doing exactly that. 

\begin{prp}
  Consider \(B \subset \real^{m+k}\) a \(\mathcal{A}\)-basic
  set such that \(\dim(\Pi_m(B)) < \dim(B)\). Then, there are \(A_1,
  \dots, A_p\) some \(\Lambda\)-sets such that
  \begin{itemize}
  \item For every \(1 \leq i \leq p\), there is some integer \(h_i
    \geq 0\) such that \(A_i \subset \real^{m+k+h_i}\), \(\dim(A_i)
    < \dim(B)\), \(\Pi_{m+k}(A_i) \subset B\) and \(\Pi_{m+k}
    \restriction A_i\) has finite fibers.
  \item There is a neighborhood \(U \subset \real^{m+k}\) of \(0\)
    such that
    \[\Pi_m(B \cap U) = \Pi_m(A_1) \cup \dots \cup \Pi_m(A_p).\]
  \end{itemize}
\end{prp}

\begin{proof}
  As said above, the proof proceeds by first parametrizing \(\mathcal{A}\)-basic sets
  using Corollary \ref{cor:best-parametrization} and then applying the Local
  Fiber Cutting Lemma to each of the charts we obtain. The following diagram
  illustrates the situation.

  
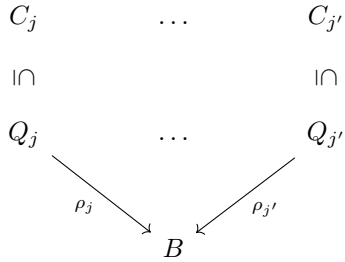
\begin{center}
    \begin{tikzcd}[column sep = large, row sep = large]
      C_{j} \arrow[phantom, sloped]{d}{\subset} &\dots &C_{j'}
      \arrow[phantom, sloped]{d}{\subset}\\
      Q_{j} \arrow[']{dr}{\rho_j} &\dots &Q_{j'} \arrow{dl}{\rho_{j'}}\\
      &B
    \end{tikzcd}
  \end{center}
  
  Consider \(\{(\rho_j, Q_j) \for j \in J\}\) a parametrization of
  \(B\) obtained by Corollary \ref{cor:best-parametrization} and fix \(j \in J\). If
  \(\dim(Q_j) < \dim(B)\) then we let \(U_j = \real^{m+k}\) and \(C_j
  = Q_j\). Notice in particular that \(\Pi_m(\rho_j(Q_j \cap U_j) ) =
  \Pi_m(\rho_j(C_j))\) and that \(\dim(C_j) < \dim(B)\). Now, assume
  that \(\dim(Q_j) = \dim(B)\). Then, let \(\eta \colon Q_j \to
  \real^m\) defined by \(\eta = \Pi_m \circ \rho_j\). By assumption,
  this map has constant rank \(l\) and we must have
  \[l = \dim(\eta(Q_j)) \leq \dim(\Pi_m(B)) < \dim(B) = \dim(Q_j).\]
  By the Fiber Cutting Lemma above, there are \(U_j \subset
  \real^{m+k}\) a neighborhood of \(0\) and \(C_j \subset Q_j\) a
  simple sub-\(\Lambda\)-set such
  that \(\eta(Q_j \cap U_j) = \eta(C_j)\) and \(\dim(C_j) <
  \dim(B)\).

  Now, consider \(J_0 \subset J\) a finite subset such that
  \(V \coloneq \bigcup_{j \in J_0} \rho_j(Q_j \cap U_j)\) is a
  neighborhood of \(0\) in \(B\). We then have
  \[\Pi_m(V) = \bigcup_{j \in J_0} \Pi_m(\rho_j(Q_j \cap U_j)) =
    \bigcup_{j \in J_0} \Pi_m(\rho_j(C_j)).\]
  For \(j \in J_0\), \(C_j\) is a simple sub-\(\Lambda\)-set so that
  \(\rho_j(C_j)\) is also a simple sub-\(\Lambda\)-set. Thus, there is
  a \(\Lambda\)-set \(A_j \subset \real^{m+k+h_j}\) such that
  \(\Pi_{m+k}(A_j) = \rho_j(C_j)\) and \(\Pi_{m+k} \restriction A_j\) has
  finite fibers. We deduce that \(\dim(A_j) = \dim(C_j)\) by Lemma
  \ref{dimension-fibers}. In
  particular, \(\dim(A_j) < \dim(B)\) for \(j \in J_0\) and
  \[\Pi_m(V) = \bigcup_{j \in J_0} \Pi_m(A_j)\]
  whence the result. 
\end{proof}

Let \(A \subset \real^{m+k}\) be a \(\Lambda\)-set. Since \(A\) is
locally a finite union of \(\mathcal{A}\)-basic set around each point
of \(\clos A\), we can use the proposition above and the compactness
of \(\clos A\) to obtain the following corollary. 

\begin{cor}[Global Fiber Cutting]
  \label{global-fiber-cutting}
  Let \(B \subset \real^{m+k}\) be a \(\Lambda\)-set such that
  \(\dim(\Pi_m(B)) < \dim(B)\). Then, there are \(\Lambda\)-sets
  \(A_1, \dots, A_p\) such that
  \begin{itemize}
  \item For every \(1 \leq i \leq p\), there is some integer \(h_i
    \geq 0\) such that \(A_i \subset \real^{m+k+h_i}\), \(\dim(A_i)
    < \dim(B)\), \(\Pi_{m+k}(A_i) \subset B\) and \(\Pi_{m+k}
    \restriction A_i\) has finite fibers.
  \item We have
    \[\Pi_m(B) = \Pi_m(A_1) \cup \dots \cup \Pi_m(A_p).\]
  \end{itemize}
  \qed
\end{cor}

\begin{remark}
  \label{rem:fibers-global-fiber-cutting}
  We require that \(\Pi_{m+k} \restriction A_i\) has finite fibers and
  that \(\Pi_{m+k}(A_i) \subset B\) so that, for every \(y \in
  \Pi_m(A_i)\), we have \(\dim((A_i)_y) \leq \dim(B_y)\). Indeed,
  consider \(y \in \Pi_m(A_i)\). We have \(\Pi_{m+k}((A_i)_y) \subset
  B_y\) and \(\Pi_{m+k} \restriction (A_i)_y\) has finite fibers so
  that, by Lemma \ref{dimension-fibers}, \(\dim((A_i)_y) =
  \dim(\Pi_{m+k}((A_i)_y)) \leq \dim(B_y)\). 
\end{remark}

We now want to prove that every sub-\(\Lambda\)-set is simple. The
rough idea of the proof is to apply the corollary above
inductively. However, we still need to make a link between
satisfying \(\dim (\Pi_m(B)) = \dim(B)\) and having \(\Pi_m
\restriction B\) have finite fibers. This is the point of the
following lemma. 

\begin{lem}
  \label{lambda-sets-constant-rank}
  Let \(B \subset \real^{m+k}\) be a \(\Lambda\)-set. Then, there are
  \(\Lambda\)-sets \(A_1, \dots, A_p\) such that
  \begin{itemize}
  \item For every \(1 \leq i \leq p\), there is some integer \(h_i
    \geq 0\) such that \(A_i \subset \real^{m+k+h_i}\) and
    \(\dim((A_i)_y)\) does not depend on \(y \in \Pi_m(A_i)\). 
  \item For each \(1 \leq i \leq p\), \(\Pi_{m+k} \restriction A_i\)
    has finite fibers and 
    \[B = \Pi_{m+k}(A_1) \cup \dots \cup \Pi_{m+k}(A_p).\]
  \end{itemize}  
\end{lem}

\begin{proof}
  By Corollary \ref{global-parametrization}, there are \(M_1, \dots,
  M_p\) some simple sub-\(\Lambda\)-sets that are also manifolds such
  that \(A = M_1 \cup \dots \cup M_p\) and \(\Pi_m \restriction M_i\)
  has constant rank for each \(1 \leq i \leq p\). Now, consider a
  \(\Lambda\)-set \(B_i \subset \real^{m+k+h_i}\) such that
  \(\Pi_{m+k}(B_i) = M_i\) and \(\Pi_{m+k} \restriction B_i\) has
  finite fibers. We have that \(A = \Pi_{m+k}(B_1) \cup \dots \cup
  \Pi_{m+k}(B_p)\). Furthermore, if \(1 \leq i \leq p\) and \(y \in
  \Pi_m(B_i)\), then \(y \in \Pi_m(M_i)\) and \((M_i)_y =
  \Pi_{m+k}((B_i)_y)\). Since \(\Pi_{m+k} \restriction (B_i)_y\) has
  finite fibers, it follows by Lemma \ref{dimension-fibers} that
  \(\dim((B_i)_y) = \dim((M_i)_y)\) whence the result. 
\end{proof}

\begin{remark}
  \label{rem:strengthen-global-fiber-cutting}
  Reusing the notations of the lemma, since \(\Pi_{m+k}(A_i) \subset
  B\) and \(\Pi_{m+k} \restriction A_i\) has finite fibers, we can
  apply Lemma \ref{dimension-fibers} to find that \(\dim(A_i) =
  \dim(\Pi_{m+k}(A_i)) \leq \dim(B)\).
  
  This allows us to refine the statement of Corollary
  \ref{global-fiber-cutting}. Indeed, we may assume that the
  \(\Lambda\)-sets \(A_1, \dots, A_p\) in the conclusion of the
  corollary satisfy that \(\dim((A_i)_y)\) does not depend on \(y \in
  \Pi_m(A_i)\) for each \(1 \leq i \leq p\). This is especially
  interesting in view of Lemma \ref{dimension-fibers}. 
\end{remark}

\begin{thm}
  \label{sub-sets-are-simple}
  Let \(A \subset \real^m\) be a sub-\(\Lambda\)-set. Then, \(A\) is
  also a simple sub-\(\Lambda\)-set. 
\end{thm}

\begin{proof}
  The following diagram illustrates the proof below. 
  
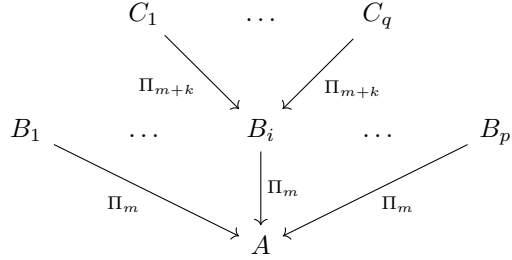
\begin{center}
    \begin{tikzcd}[row sep = large]
      &C_1 \arrow[']{dr}{\Pi_{m+k}} &\dots &C_q \arrow{dl}{\Pi_{m+k}} \\
      B_1 \arrow[']{drr}{\Pi_m} &\dots &B_i \arrow{d}{\Pi_m} &\dots
      &B_p \arrow{dll}{\Pi_m} \\
      &&A
    \end{tikzcd}
  \end{center}
  
  It suffices to show that there are \(\Lambda\)-sets \(B_1, \dots,
  B_p\) such that \(\Pi_m \restriction B_i\) has finite fibers and \(A
  = \Pi_m(B_1) \cup \dots \cup \Pi_m(B_p)\). Thus, consider \(B_1,
  \dots, B_p\) some \(\Lambda\)-sets such that \(A = \Pi_m(B_1) \cup
  \dots \cup \Pi_m(B_p)\). By applying Lemma
  \ref{lambda-sets-constant-rank}, we can assume
  that, for every \(1 \leq i \leq p\), the dimension \(\mu_i \coloneq
  \dim((B_i)_y)\) does not depend on \(y \in \Pi_m(B_i)\). Define also
  \(d_i = \dim(\Pi_m(B_i))\) for each \(1 \leq i \leq p\). Finally,
  consider \(\mu = \max(\mu_1, \dots, \mu_p)\) and \(d = \max\{d_i \for \mu_i
  = \mu\}\) and assume that \(B_1, \dots, B_p\) have been chosen in such
  a way that the pair \((\mu, d)\) is minimal for the lexicographic
  order. To prove the result, we only need to show that \(\mu =
  0\). Thus, assume by contradiction that \(\mu > 0\).

  Let \(1 \leq i \leq p\) such that \(\mu_i = \mu\) and \(d_i =
  d\). For the rest of the proof, we will write \(B = B_i\). By Lemma
  \ref{dimension-fibers}, we have \(\dim(B) = d + \mu > d\) so
  that, by Corollary \ref{global-fiber-cutting}, there are \(C_1,
  \dots, C_{q}\) some \(\Lambda\)-sets such that
  \[\Pi_m(B) = \Pi_m(C_1) \cup \dots \cup \Pi_m(C_{q})\]
  and \(\dim (C_j) < \dim(B)\) for \(1 \leq j \leq q\). Up to
  strengthenning Corollary \ref{global-fiber-cutting} as explained in
  Remark \ref{rem:strengthen-global-fiber-cutting}, we may
  assume that, for every \(1 \leq j
  \leq q\), the dimension \(\mu_j \coloneq \dim((C_j)_y)\) does not
  depend on \(y \in \Pi_m(C_j)\). Define
  also \(d_j = \dim(\Pi_m(C_j))\) for \(1 \leq j \leq q\). Then,
  fixing \(1 \leq j \leq q\),  we have \(\mu_j \leq \mu\) by Remark
  \ref{rem:fibers-global-fiber-cutting}. 
  If \(\mu_j = \mu\), then \(\dim(C_j) =
  \mu_j + d_j = \mu + d_j\) so that \(\mu + d > \mu+d_j\) whence
  \(d_j < d\). Thus, we contradict the minimality of the pair \((\mu,
  d)\) by
  replacing \(B\) with the family \(C_1, \dots, C_q\) and by repeating
  the same operation for each \(B_i\) such that \(\mu_i = \mu\) and
  \(d_i = d\). 
\end{proof}

\begin{remark}
  The proof can be seen as a procedure that takes a representation of
  the form \(A = \Pi_m(B_1) \cup \dots \cup \Pi_m(B_p)\) and that
  produces a new one such that the pair \((\mu, d)\) decreases. Since
  this procedure can be applied as long as \(\mu > 0\), we can apply
  it repeatedly until \(\mu = 0\). Also, the set of pairs \((\mu, d)\)
  is well ordered so that the process must stop eventually. 
\end{remark}

It is easy to see that \(\Lambda\)-sets have the \(\Lambda\)-Gabrielov
property using Theorem \ref{sub-sets-are-simple} and Corollary 
\ref{global-parametrization}. Following \cite[Theorem
4.18]{guenet-weakly-smooth}, we obtain:

\begin{thm}
  \label{thm:o-minimal}
  The structure \(\real_{\mathcal{A}}\) is o-minimal and
  model-complete. Furthermore, the sub-\(\Lambda\)-sets are exactly
  the bounded definable sets. 
\end{thm}

\section{A nowhere smooth function definable in an o-minimal
  structure}
\label{sec:example}

In this final section, we are going to give an application of the
results proven above by constructing an o-minimal structure in which
we can define a nowhere smooth function. To do so, we borrow
techniques from \cite{legal-generic} and
\cite{legal-rolin-not-c-infty}. In Paragraph \ref{par:complete-tvs},
we construct a complete topological vector space of functions. Then,
in Paragraph \ref{par:strongly-transcendental}
we use Baire's Category Theorem to show that the set of strongly
transcendental functions (see Definition \ref{defn:strongly-transcendental})
is dense in the topological vector space. It turns out that such
strongly transcendental functions are nowhere smooth in the interval
\((-1, 1)\). Finally, in Paragraph
\ref{par:strongly-transcendental-omin}, we show that strongly
transcendental functions generate o-minimal structures. 

\subsection{A complete topological vector space}
\label{par:complete-tvs}

Throughout this section, we fix \((a_i)_{i \in \nat}\) a sequence of
points in the interval \((-1, 1)\) and we assume that the set \(S =
\{a_i \for i \in \nat\}\) is dense in \((-1, 1)\). We will be
considering the following vector space of functions. 

\begin{defn}
  We let \(\almostsmooth\) be the vector space of continuous functions \(f \colon
  (-1, 1) \to \real\) such that:
  \begin{itemize}
  \item The germ of \(f\) at any point \(x \in (-1, 1) \setminus S\) is
    weakly smooth.
  \item For \(p \in \nat\), the germ of \(f\) at \(a_p\) is \(C^p\). 
  \item For every \(a \in S\) and every \(p \geq 0\), the two limits
    \(\lim\limits_{x \to a^-} f^{(p)}(x)\) and \(\lim\limits_{x \to
        a^+} f^{(p)}(x)\) exist. 
  \end{itemize}
\end{defn}

\begin{remark}
  Reusing the notation of the definition, recall that, by definition
  of weakly smooth germs, the derivative
  \(f^{(p)}(x)\) exists for every \(x \in (-1, 1) \setminus S\) and
  every \(p \in \nat\). Thus, \(f^{(p)}(x)\) is well-defined for every
  \(x \in (-1, 1) \setminus \{a_0, \dots, a_{p-1}\}\) which is a dense
  open set in \((-1, 1)\) so that the third point in the definition
  above makes sense. 
\end{remark}

\begin{defn}
  Let \(f \in \almostsmooth\) and \(p \in \nat\). Then, we define the
  map
  \[(-1, 1) \ni x \mapsto f^{(p)}(x^-) \hspace{10pt}\text{(resp. }
    (-1, 1) \ni x \mapsto  f^{(p)}(x^+))\]
  to be the extension of the map
  \((-1, 1) \setminus \{a_0, \dots, a_{p-1}\} \ni x \mapsto
  f^{(p)}(x)\) by continuity on
  the left. (resp. continuity on the right).
\end{defn}

We are now going to define a topology on \(\almostsmooth\) by using a
family of seminorms. 

\begin{prp}
  Let \(K \subset (-1, 1)\) be a compact set and \(p \geq 0\) be an
  integer. We define a seminorm on \(\almostsmooth\) by letting
  \[\norm f {K, p} = \sup\left\{\abs{f^{(q)}(x)} \for x \in K
      \setminus S, q \leq p\right\},\]
  where the supremum is taken to be \(0\) whenever \(K \subset S\). 
  Furthermore, the family of all of these seminorms is separating.
\end{prp}

\begin{proof}
  We must show that \(\norm{\ }{K, p}\) is always finite. To do so,
  fix \(f \in \almostsmooth\) and assume that \(\norm f {K, p} = \infty\). There
  must be \(q \leq p\) and a sequence \((x_n)_{n \in \nat}\) of points
  of \(K \setminus S\) such that \(\abs{f^{(q)}(x_n)} \to
  \infty\). Up to extracting a subsequence, we may assume that there
  exists \(x \in K\) with \(x_n \to x\) and that, either \(x_n \leq x\)
  for all \(n \in \nat\) or \(x_n \geq x\) for all \(n \in
  \nat\). Since both cases are handled similarly, we assume that \(x_n
  \leq x\) for all \(n \in \nat\). Since the map \(K \ni y \mapsto
  f^{(q)}(y^-)\) is continuous on the left, we now have \(f^{(q)}(x_n)
  \to f^{(q)}(x^-)\) which is contradictory so that \(\norm{\ }{K,
    p}\) is always finite. It is now easy to see that this is a
  seminorm. Furthermore, since \((-1, 1) \setminus S\) is dense in
  \((-1, 1)\) and all functions in \(\almostsmooth\) are continuous,
  it is clear that the family of seminorms we have defined is separating. 
\end{proof}

\begin{defn}
  We make \(\almostsmooth\) into a locally convex topological vector space by
  endowing it with the topology generated by the family of seminorms
  \(\norm{\ }{K, p}\) for \(K \subset (-1, 1)\) a compact set and \(p
  \in \nat\). 
\end{defn}

Notice that the topology on \(\almostsmooth\) is metrisable. We would
now like to prove that \(\almostsmooth\) is a Baire space. To do so,
we use the Baire Category Theorem and the following proposition. 

\begin{prp}
  The topological vector space \(\almostsmooth\) is complete. 
\end{prp}

\begin{proof}
  Since \(\almostsmooth\) is metrisable, it suffices to show that any Cauchy
  sequence \((f_n)_{n \in \nat}\) in \(\almostsmooth\) converges to some element \(f \in
  \almostsmooth\). To do so, fix an integer \(p \geq 0\) and let \(g_{p, n} \colon
  (-1, 1) \to \real\) be the function defined by \(g_{p, n}(x) =
  f_n^{(p)}(x^-)\). For \(\epsilon > 0\), we let \(K_\epsilon =
  [-1+\epsilon, 1-\epsilon]\). For \(n, m \in \nat\) and \(x \in
  K_\epsilon \setminus S\), we have
  \[\abs{g_{p,n}(x) - g_{p, m}(x)} = \abs{f_n^{(p)}(x) - f_m^{(p)}(x)}
    \leq \norm f{K_\epsilon, p}.\]
  Furthermore, each \(g_{p, n}\) is left-continuous and \(S\) is
  countable. Thus, assuming that \(-1+\epsilon \not \in S\), we get
  that \(\abs{g_{p, n}(x) - g_{p, m}(x)} \leq \norm f {K_\epsilon,
    p}\) for all \(x \in K_\epsilon\). Since \(\bigcup_{\epsilon > 0}
  K_\epsilon = (-1, 1)\) it follows that the sequence \((g_{p, n})_{n
    \in \nat}\) converges to a function \(g_p \colon (-1, 1) \to
  \real\) uniformly on all compacts \(K \subset (-1, 1)\). Similarly,
  letting \(h_{p, n}
  \colon (-1, 1) \to \real\) be the function defined by \(h_{p, n}(x)
  = f_n^{(p)}(x^+)\) for all \(x \in (-1, 1)\), we find that the
  sequence \((h_{p, n})\) converges to a function \(h_p \colon (-1, 1)
  \to \real\) uniformly on all compacts \(K \subset (-1, 1)\).

  For \(n \in \nat\), the function \(f_n\) is continuous so that
  \(g_{0, n} = h_{0, n}\). Thus, \(g_0 = h_0\) and we let \(f = g_0 =
  h_0\). We are going to show that \(f \in \almostsmooth\) and that \(f_n \to
  f\). To begin with, \(g_p\) is left-continuous and \(h_p\) is
  right-continuous for all \(p \in \nat\). Furthermore, let \(b_0,
  \dots, b_{p-1}\) be a reordering of the elements \(a_0, \dots,
  a_{p-1}\) such that \(b_0 < \dots < b_{p-1}\). Define also \(b_{-1}
  = -1\) and \(b_p = 1\). For each \(0 \leq i \leq p\), each \(f_n\)
  is \(C^p\) on the interval \(I_i\) by assumption. Thus, \(f
  \restriction I_i\) is also \(C^p\) on \(I_i\) and we have
  \[g_p \restriction I = (f \restriction I)^{(p)} = h_p \restriction
    I.\]
  Thus, we conclude at once that \(f \in \almostsmooth\) and that
  \(g_p(x) = f^{(p)}(x^-)\) and \(h_p(x) = f^{(p)}(x^+)\) for all \(p
  \in \nat\) and all \(x \in (-1, 1)\). It is thus clear that \(f\) is
  the limit of the sequence \((f_n)\) which concludes. 
\end{proof}

By applying the Baire Category Theorem, we now get the following
corollary. 

\begin{cor}
  The topological vector space \(\almostsmooth\) is a Baire
  space. \qed
\end{cor}

\subsection{Strongly transcendental functions in \(\almostsmooth\)}
\label{par:strongly-transcendental}

Recall that a subset \(A \subset \almostsmooth\) is called \emph{residual}
whenever it is a countable intersection of dense open sets and that
such residual sets are dense.

\begin{defn}
  Let \(f \in \almostsmooth\) be a function, \(n, p, k \in \nat\) be
  integers such that \(p \geq k\) and \(x = (x_1, \dots, x_n)\) be a
  tuple of points in \((-1, 1)
  \setminus \{a_0, \dots, a_{p-1}\}\). We then define \(j_n^{p,
    k}f(x)\) to be the tuple containing the following elements (with
  multiplicity):
  \begin{itemize}
  \item \(f^{(q)}(x_i)\) for \(0 \leq q \leq p\) and \(1 \leq i \leq
    n\).
  \item \(f^{(q)}(a_i)\) for \(0 \leq q \leq i\) and \(0 \leq i \leq
    k\).
  \item \(f^{(q)}(a_i^-), f^{(q)}(a_i^+)\) for \(i < q \leq p\) and
    \(0 \leq i \leq k\). 
  \end{itemize}
\end{defn}

\begin{defn}
  \label{defn:strongly-transcendental}
  Given \(f \in \almostsmooth\), we say that \(f\) is \emph{strongly
    transcendental} when, for all integers \(n, k \in \nat\) and
  all tuples \(x = (x_1, \dots, x_n)\) of points of \((-1, 1)
  \setminus S\) such that \(x_i \neq x_j\) whenever \(i \neq j\),
  there exists \(C \in \nat\) such that
  \[\trdeg(x, a_0, \dots, a_k, j_n^{p, k}f(x)) \geq n(p+2) +
    (k+1)(2p+3) - \frac{(k+1)(k+2)}{2} - C\]
  for all \(p \in \nat\) with \(p \geq k\). 
\end{defn}

\begin{remark}
  Let \(f \in \almostsmooth\) be strongly transcendental. Then, for each \(a \in
  S\), the germ of \(f\) at \(a\) is not weakly smooth. Since \(S\) is
  dense in \((-1, 1)\), it follows that \(f\) is nowhere smooth. 
\end{remark}

The rest of this paragraph is dedicated to proving the following
theorem. 

\begin{thm}
  \label{thm:strongly-transcendental-residual}
  The set \(A \subset \almostsmooth\) of strongly transcendental
  functions is dense in \(\almostsmooth\). 
\end{thm}

\begin{remark}
  Reusing the notation of Definition
  \ref{defn:strongly-transcendental}, we will actually show a
  stronger result according to which the constant \(C\) can
  generically (meaning in a residual set) be taken as \(C =
  n+k+1\). 
\end{remark}

\begin{lem}
  Fix \(n, p, k \in \nat\) some integers such that \(p \geq k\) and
  let \(X\) be a \(\rational\)-algebraic set of codimension \(n+1\) in
  \(\real^{l}\) with
  \[l = n(p+2) + 2(k+1)(p+1) - \frac{(k+1)(k+2)}{2}.\]
  If \(D = (-1, 1) \setminus \{a_1, \dots, a_{p}\}\) and
  \(\Delta_n = \{x \in D^n \for \exists i \neq j,
  x_i = x_j\}\) then the set
  \[E_n^{p, k}(X) = \{f \in \almostsmooth \for \forall x \in D^n
    \setminus \Delta_n, (x, j_n^{p,k}f(x)) \not \in X\}\]
  is residual. 
\end{lem}

\begin{proof}
  To begin with, we are going
  to prove that \(E_n^{p, k}(X)\) is dense in \(\almostsmooth\). To do
  so, let \(f \in \almostsmooth\)
  and consider \(b_0, \dots, b_k\) the permutation of \(a_0, \dots,
  a_k\) such that \(b_0 < \dots < b_k\). Define also \(b_{-1} = -1\)
  and \(b_{k+1} = 1\). Let \(\epsilon_{ij} \in \real\) for \(0 \leq i
  \leq k+1\) and \(0 \leq j \leq (n+2)(p+1)-1\). We then define the
  function \(h_\epsilon \colon (-1, 1) \to \real\) by letting
  \begin{equation*}
    h_\epsilon(x) = \epsilon_{i, 0} + \dots + \epsilon_{i,
      (n+2)(p+1)-1} x^{(n+2)(p+1)-1}
  \end{equation*}
  for \(x \in [b_{i-1}, b_i)\). The vector subspace
  \[W = \{\epsilon \in \real^{(k+2)(n+2)(p+1)} \for h_\epsilon \in
    \almostsmooth\}\]
  has codimension \(\frac{(k+1)(k+2)}{2}\).
  We can then define a map
  \begin{alignat*}{3}
    \Phi \colon D^n \setminus \Delta_n \times W &\to \real^{l} \\ 
    (x, \epsilon) &\mapsto (x, j_n^{p, k}(f+h_\epsilon)(x)).
  \end{alignat*}
  We are now going to prove that this map is \(C^1\) and that it is a
  submersion. \\

  To do so, let \(x \in D^n \setminus \Delta^n\), \(\epsilon, \delta \in
  W\), \(0 \leq q \leq p\), \(1 \leq i \leq n\) and \(y \in D^n
  \setminus \Delta^n\) such
  that the segment from \(x\) to \(x+y\) is fully included in
  \(D^n \setminus \Delta_n\). We have
  \begin{align*}
    (f + h_{\epsilon + \delta})^{(q)}(x_i + y_i)
    &= f^{(q)}(x_i+y_i) + h_{\epsilon}^{(q)}(x_i + y_i) +
      h_{\delta}^{(q)}(x_i+y_i) \\
    &= f^{(q)}(x_i) + y_if^{(q+1)}(x_i) + o(y_i) +
      h_{\epsilon}^{(q)}(x_i)\\
    &\hspace{10pt} + y_ih_{\epsilon}^{(q+1)}(x_i) + o(y_i)
      + h_\delta^{(q)}(x_i) + o(\delta).
  \end{align*}
  Thus, the map \((x, \epsilon) \mapsto (f+h_{\epsilon})^{(q)}(x_i)\)
  is differentiable and if we let \(\grad^0_{i,q}\) be its gradient, then
  we find
  \[\langle \grad^0_{i, q}(x, \epsilon) \scalar (y, \delta) \rangle =
    y_i[f^{(q+1)}(x) + h_{\epsilon}^{(q+1)}(x)] + h_\delta^{(q)}(x_i).\]
  Notice that \(\grad^0_{i, q}\) is continuous.

  Furthermore, if \(x, y, \epsilon, \delta, q\) are as above, \(0 \leq
  i \leq k\) and \(s \in \{+, -\}\) is a sign, then we can show as
  above that the map \((x, \epsilon) \mapsto
  (f+h_\epsilon)^{(q)}(a_i^s)\) is differentiable. Furthermore, its
  gradient \(\grad^{1, s}_{i, q}\) is such that
  \[\langle \grad^{1, s}_{i, q}(x, \epsilon) \scalar (y, \delta) \rangle =
    h_\delta^{(q)}(a_i^s).\]
  Notice that \(\grad^{1, s}_{i, q}\) does not depend on the sign
  \(s\) whenever \(q \leq i\) and in this case, we let \(\grad^1_{i,
    q}\) be their common value. 

  In order to conclude the proof that the map \(\Phi\) is a
  submersion, we must prove that the following family of gradients is
  linearly independent.
  \begin{itemize}
  \item \(\grad^0_{i, q}\) for \(1 \leq i \leq n\) and \(0 \leq q \leq
    p\). 
  \item \(\grad^1_{i, q}\) for \(0 \leq i \leq k\) and \(0 \leq q \leq
    i\).
  \item \(\grad^{1, s}_{i, q}\) for \(0 \leq i \leq k\), \(i <  q \leq
    p\) and \(s \in \{+, -\}\). 
  \item \(\grad(x_i)\) for \(1 \leq i \leq n\). 
  \end{itemize}
  But it is easy to see that the common zero set of these linear forms
  is reduced to \(0\), whence the result. \\

  Finally, since \(X\) is an algebraic set of codimension at most
  \(n+1\), it is a union of finitely many smooth manifolds of
  codimension at least \(n+1\). Thus, \(\Phi^{-1}(X)\) is the union of
  finitely many \(C^1\)-submanifolds of codimension at least
  \(n+1\). Thus, by Sard's Theorem, \(\pi(\Phi^{-1}(X))\) is the union
  of finitely many sets of measure \(0\), and hence has measure \(0\),
  where \(\pi \colon D^n \setminus \Delta_n \times W \to W\) is the
  projection \(\pi(x, \epsilon) = \epsilon\). This proves that there
  are \(\epsilon \in W\) as small as we wish such that \(f +
  h_\epsilon \in E_n^{p, k}(X)\) whence the density.

  The only part of the proof left is to show that \(E_n^{p, k}(X)\) is
  the intersection of countably many open sets. To do so, let
  \((K_m)_{m \geq 0}\)
  be an exhaustion of \(D^n \setminus \Delta_n\) by compact
  sets. Then, define
  \[E_{n, m}^{p, k}(X) = \{f \in \almostsmooth \for \forall x \in K_m,
    (x, j_n^{p, k}f(x)) \not \in X\}.\]
  We have \(E_n^{p, k}(X) = \bigcap_{m \geq 0} E_{n, m}^{p, k}(X)\) so
  that it suffices to prove that each \(E_{n, m}^{p, k}(X)\) is
  open. To do so, assume that \((f_\gamma)_{\gamma \geq 0}\) is a
  sequence in its complement which converges to \(f \in
  \almostsmooth\). For each \(\gamma \geq 0\), there is 
  \(x_\gamma \in K_m\) such that \((x_\gamma, j_n^{p,
    k}f_\gamma(x_\gamma)) \in X\). By compactness of \(K_m\), there
  exists \(x \in K_m\) such that \(x_\gamma \to x\). Then, it is clear
  that we have \((x_\gamma, j_n^{p, k}f_\gamma(x_\gamma)) \to (x,
  j_n^{p, k}f(x))\). Since \(X\) is closed, it follows that \((x,
  j_n^{p, k}f(x)) \in X\) whence \(f \in E_{n, m}^{p, k}\) which
  concludes. 
\end{proof}

\begin{proof}[Proof of Theorem \ref{thm:strongly-transcendental-residual}]
  It suffices to show that
  \[A \supset \bigcap_{n, p, k, X} E_n^{p, k}(X),\]
  where \(n, p, k\) range over \(\nat\) and \(X\) ranges over
  \(\rational\)-algebraic subsets of codimension \(n+1\) in
  \(\real^l\), with
  \[l = l(n, p, k) =  n(p+2) + 2(k+1)(p+1) - \frac{(k+1)(k+2)}{2}.\]
  To do so, let \(f \in \almostsmooth\) be an element contained in the
  right hand side and let \(n, k \in \nat\) and \(x = (x_1, \dots,
  x_n)\) be a tuple of points in \((-1,1) \setminus S\). Assume that
  \(x_i \neq x_j\) whenever \(i \neq j\). Then, given \(p \in \nat\)
  such that \(p \geq k\), let \(X\) be a
  \(\rational\)-algebraic subset of \(\real^l\) with least dimension
  such that \((x, j_n^{p, k} f(x)) \in X\). Then, \(X\) has
  codimension at most \(n\). Thus
  \[\trdeg(x, j_n^{p, k}f(x)) \geq l-n\]
  from which it follows that
  \[\trdeg(x, a_0, \dots, a_k, j_n^{p, k}f(x)) \geq n(p+2) +
    (k+1)(2p+3) - \frac{(k+1)(k+2)}{2} - n - k - 1.\]
  Since this holds for all \(p \in \nat\), we obtain that \(f\) is
  strongly transcendental at once. 
\end{proof}

\subsection{An o-minimal structure}
\label{par:strongly-transcendental-omin}

Throughout this paragraph, we let \(f \colon (-1, 1) \to \real\) be a
strongly transcendental function in \(\almostsmooth\). We are going to
show that the structure \(\langle \real_\text{alg}; f \rangle\) is
o-minimal. To do so, we begin by constructing a quasianalytic class of
functions. Firstly, recall the following

\begin{lem}[\cite{legal-rolin-not-c-infty}, Lemma 2.1]
  Let \(F \in \formser \real X\), where \(X\) is a single
  variable. There exists a function \(f \colon \real \to \real\) such
  that
  \begin{itemize}
  \item The germ of \(f\) at the origin is weakly smooth.
  \item The restriction of \(f\) to the complement of any neighborhood
    of \(0\) is given piecewise by finitely many polynomials. 
  \item We have \(\taylor(f) = F\).
  \end{itemize}
\end{lem}

For each \(x \in (-1, 1) \setminus S\), we let \(f_x\) be the germ at
the origin of the function \(y \mapsto f(x + y)\). Now, let \(a \in
S\) and let \(g_{a} \colon \real \to \real\) be a function as in the
previous lemma, with
\[\taylor(g_a) = \sum_{q \geq 0} \frac{f^{(q)(a^+)}}{q!} X^q.\]
We now let \(f_{a^+}\) be the germ at \(0\) such that
\begin{equation*}
  f_{a^+}(y) =
  \begin{cases}
    f(a+y) &\text{if } y \geq 0 \\
    g_a(y) &\text{otherwise.}
  \end{cases}
\end{equation*}
It is clear that the germ \(f_{a^+}\) is weakly smooth. Similarly, we
can define a germ \(f_{a^-}\) at the origin such that \(f_{a^-}(y) =
f(a+y)\) when \(y \leq 0\). 

\begin{remark}
  Notice that the germ \(f_{a^+}\) depends on the choice of
  \(g_{a}\). Similarly, the germ \(f_{a^-}\) will depend on the choice
  of a function \(h_a\). Also, each germ \(f_x\) for \(x \in (-1, 1)
  \setminus S\) and \(f_{a^+}, f_{a^-}\) for \(a \in S\) have canonical
  representatives. In what follows, we will identify the germs and
  their representatives.
\end{remark}

We then let \(\mathcal{A}\) be the class of all germs that can be
written as
\(\mathcal{L}(g_1, \dots, g_m)\) where \(\mathcal{L}\) is an operator
as defined in Section 2.2 of \cite{legal-rolin-not-c-infty} and \(g_1,
\dots, g_m\) are distinct members of \(\mathcal{G}\), where
\[\mathcal{G} = \{f_x \for x \in (-1, 1) \setminus S \cup \{a^+ \for a
  \in S\} \cup \{a^- \for a \in S\}.\]
Arguing as in Lemma 3.6 of \cite{legal-generic} and using that \(f\)
is strongly transcendental, we can show that the algebras
\(\mathcal{A}_n\) are quasianalytic for each \(n \geq
1\). Furthermore, these algebras obviously contain polynomials and are
stable under composition, monomial division and implicit functions.

In order to apply Theorem \ref{thm:o-minimal}, it suffices to check
that each germ
\(g  \in \mathcal{A}_n\) has a representative (which we still denote by
\(g\)) such that \(\graph g\) is a simple sub-\(\Lambda\)-set. We
start by considering the special case when \(g \in \mathcal{G}\) and
we let \(y_0\) be sufficiently close to \(0\). Then there exist \(h_1,
h_2 \in \mathcal{G}\) such that the following is an equality between
germs of sets at \(y_0\):
\[\graph g = \{(y, z) \for y \geq y_0, z = h_1(y - y_0)\} \cup \{(y,
  z) \for y \leq 0, z = h_2(y - y_0)\}.\]
Finally, the general case follows from the special one by the same
technique as in the proof of Lemma 2.5 in
\cite{guenet-weakly-smooth}. Thus, we can apply Theorem
\ref{thm:o-minimal} to obtain that

\begin{prp}
  The structure \(\real_\mathcal{A}\) is o-minimal and it has the same
  definable sets as \(\langle \real_\text{alg}; f\rangle\). 
\end{prp}

Finally, let \(\phi \colon \real \to (-1, 1)\) be the function defined
by
\[\phi(x) = \frac{x}{\sqrt{1 + x^2}}.\]
Since the function \(\phi\) is an \(\real_\mathcal{A}\)-definable
homeomorphism, it follows that the function \(f \circ \phi\) is a
nowhere smooth function which is \(\real_\text{A}\)-definable.

\end{document}